\documentclass[11pt]{amsart}
\usepackage{amssymb,amsmath,amsfonts,amscd,euscript,indentfirst}
\usepackage{tikz-cd}

\usepackage[backref=page]{hyperref}

\newtheorem{theoremA}{Theorem}

\numberwithin{equation}{section}
\newtheorem{thm}{Theorem}[section]
\newtheorem{prop}[thm]{Proposition}
\newtheorem{lem}[thm]{Lemma}
\newtheorem{cor}[thm]{Corollary}
\newtheorem{rem}[thm]{Remark}
\newtheorem{example}[thm]{Example}
\newtheorem{dfn}[thm]{Definition}

\newcommand{\la}{\lambda}
\newcommand{\al}{\alpha }
\newcommand{\be}{\beta }

\newcommand{\om}{\omega }

\newcommand{\T}{\otimes}
\newcommand{\Hom}{\mathrm{Hom}}

\newcommand{\Gr}{\mathrm{Gr}}

\newcommand{\Proj}{\mathrm{Proj}}
\newcommand{\Col}{\mathrm{Col}}
\newcommand{\G}{\mathrm{G}}

\newcommand{\msl}{\mathfrak{sl}}
\newcommand{\mgl}{\mathfrak{gl}}
\newcommand{\U}{\mathrm U}

\newcommand{\bC}{{\mathbb C}}
\newcommand{\bK}{{\mathbb K}}

\newcommand{\bZ}{{\mathbb Z}}

\newcommand{\bP}{{\mathbb P}}
\newcommand{\bA}{{\mathbb A}}
\newcommand{\bG}{{\mathbb G}}

\newcommand{\fp}{{\mathfrak p}}
\newcommand{\fg}{{\mathfrak g}}
\newcommand{\fb}{{\mathfrak b}}
\newcommand{\fh}{{\mathfrak h}}
\newcommand{\fn}{{\mathfrak n}}
\newcommand{\fa}{{\mathfrak a}}

\newcommand{\fl}{{\mathfrak l}}

\newcommand{\eO}{\EuScript{O}}
\newcommand{\eL}{\EuScript{L}}
\newcommand{\eU}{\EuScript{U}}
\newcommand{\eI}{\EuScript{I}}
\newcommand{\eF}{\EuScript{F}}
\newcommand{\eT}{\EuScript{T}}

\newcommand{\br}{{\bf r}}
\newcommand{\bs}{{\bf s}}
\newcommand{\Bl}{\mathrm{Bl}}

\begin{document}

\title[Truncated Grassmannians, blow-ups and collineations]
{Truncated Grassmannians, blow-ups along Schubert varieties and 
collineations}

\author{Evgeny Feigin}
\address{Evgeny Feigin:\newline
	School of Mathematical Sciences, Tel Aviv University, Tel Aviv, 69978, Israel}
\email{evgfeig@gmail.com}

\begin{abstract}
Truncated Grassmannians are defined as closures of orbits of abelian unipotent 
groups acting on the degree  truncations of projectivized wedge powers. 
We show that such truncations in a more general setup show up in the description
of the blow-ups of general flag varieties along Schubert subvarieties. 
We work out the case of Grassmannians in detail.
In particular, we show that our blow-ups are members of a larger family of varieties 
projecting onto Grassmannians, and describe the  fibers of these projections 
via the spaces of  collineations.     		
\end{abstract}

\maketitle

\section*{Introduction}
Let $\fg$ be a simple Lie algebra over the filed of complex numbers and let $G$ be the 
corresponding simple simply connected Lie group. For a standard parabolic subgroup $P$,
let $G/P$ be a generalized flag variety.
We  denote by $S_\sigma$ the Schubert variety in $G/P$ attached to the Weyl group element $\sigma$. 
In this paper we describe 
the blow-ups along $S_\sigma$ for certain Schubert varieties in the complex Grassmannians
via truncated Grassmannians. 
Not much is known about the blow-ups of flag varieties along Schubert subvarieties in general,
although some partial results (mostly in the smooth case) 
are available \cite{BS25,HKLS25,KP21,LSi21}. 
We formulate a general approach, based on the truncated flag varieties,
and work out the construction in detail for the Grassmannians. We provide more details below.

Let $V$ be a vector space of dimension $n$ and let us fix $d<n$; 
to simplify the notation, we assume that $d\le n-d$.
We fix 
a decomposition $V=V_d\oplus V_{n-d}$ fixing two subspaces of $V$ of dimensions 
$d$ and $n-d$.
Let 
$\Gr_d(V)$ be the Grassmannian of $d$-dimensional subspaces in $V$, which 
can be realized as the quotient 
$SL_n/P_d$, where $P_d$ is the maximal parabolic subgroup corresponding to the
$d$-th simple root; in particular, $P_d$ preserves $V_d$. Let $\fp_d$ be the Lie 
algebra of $P_d$ and let $\fa_d$ 
be the opposite abelian unipotent radical, i.e. $\fa_d\oplus\fp_d = \msl(V)$.
The group $\exp(\fa_d)$ acts on the Grassmannian and $\Gr_d(V)$ is the closure 
of the orbit $\exp(\fa_d)V_d$.

The Grassmannian $\Gr_d(V)$ admits the Pl\"ucker embedding 
$\Gr_d(V)\subset \bP(\Lambda^d V)$ \cite{Fu97}. Let us fix a basis $v_1,\dots,v_n$ of $V$ such that
the first $d$ vectors form a basis of $V_d$ and the last $n-d$ form a basis of $V_{n-d}$.
The wedge power $\Lambda^d V$ admits the standard basis of the form $v_I$ consisting of
wedge monomials, where $I$ is a cardinality $d$ subset of the set $\{1,\dots,n\}$.
We define the degree $\deg I$ as the number of elements of $I$ that are larger than $d$.
For a number $r=0,\dots,d$, we introduce the space $(\Lambda^dV)_r$, which is the 
quotient of $\Lambda^dV$ by the space spanned by all vectors $v_I$ of degree larger than 
$r$. Then  $(\Lambda^dV)_r$ inherits the cyclic action of $\fa_d$ with the cyclic vector $w_r$, 
the image of the cyclic vector of $\Lambda^dV$. 
 
By definition, the truncated Grassmannian $X_r$ is a subvariety of $\bP((\Lambda^dV)_r)$
defined as the closure of the $\exp(\fa_d)$ orbit of the line $[w_r]$. 
One easily sees that $X_r$ admits an action of a larger group $P_d^-$, where 
$P_d^-$ is the (opposite) maximal parabolic subgroup attached to the $d$-th simple root.
In particular,
$X_1$ is the projective space $\bP^{d(n-d)}$ and $X_d=\Gr_d(V)$. By definition, all
$X_r$ are pairwise birationally isomorphic $\bG_a^{d(n-d)}$ varieties \cite{A11,HT99,Fe12}. 
The closure
of the graph of the birational isomorphism $X_1\to X_d$ was studied in \cite{FSS25,BSS25}
(see also \cite{Fe24,Fe25}). Our goal is to study the graph closures 
of maps between various truncated
Grassmannians and to provide a link with the blow-ups of the classical Grassmann varieties.

For $r=1,\dots,d$ let $S_r\subset \Gr_d(V)$ be the Schubert subvariety whose points $U$ 
satisfy the condition $\dim U\cap V_{n-d}\ge r$. In particular, the codimension of 
$S_r$ is $r^2$ and $S_r$ exhaust all the 
$P_d^-$ invariant Schubert varieties. Our first theorem is as follows.   
  
\begin{theoremA}
The blow-up $\Bl_{S_r}\Gr_d(V)$ is isomorphic to the closure of the graph 
of the birational map $X_{r-1}\to\Gr_d(V)$. In particular, $\Bl_{S_r}\Gr_d(V)$ admits 
the action of $\exp(\fa_d)$ with an open dense orbit. 	
\end{theoremA} 
 
We can restate the theorem in the following form: $\Bl_{S_r}\Gr_d(V)$ admits
a closed embedding into the product $\Gr_d(V)\times X_{r-1}$ with the image being the closure
of the $\exp(\fa_d)$ orbit of the point $V_d\times [w_{r-1}]$. 
This is a particular case of a general statement (Theorem \ref{thm:gen}); we give a 
separate proof in the case of Grassmannians since here we can explicitly compute 
the Koszul type resolution of the ideal sheaf $\eI_{S_r}$.

All the varieties $X_r$ and  $\Bl_{S_r}\Gr_d(V)$ are birationally isomorphic with the
isomorphism provided by the "common" open $\exp(\fa_d)$ orbit. We define a larger family
of pairwise birationally isomorphic $\bG_a^{d(n-d)}$ varieties,  
which includes all the above examples. More precisely, let $\br$ be a collection of 
numbers $1<r_1<\dots <r_m\le d$. We define $\Bl_\br$ as the multi-projective blow-up 
with respect to the varieties $S_{r_1},\dots, S_{r_m}$. 
$\Bl_\br$ can be also described as the closure 
of the $\exp(\fa_d)$ orbit through the point 
$V_d\times [w_{r_1-1}]\times\dots\times [w_{r_m-1}]$ inside the product of $\Gr_d(V)$ and
the corresponding truncated Grassmannians $X_{r_i-1}$. To describe the fibers of the natural
projection map  
$\Bl_\br\to\Gr_d(V)$ we consider the partial collineations $\Col_\bs$ \cite{L88,LAT89,Vain84}. 

Let $E_1,E_2$ be two spaces of dimension $N$ and let $\bs=(1\le s_1<\dots <s_m<N)$.
Then $\Col_\bs$ is a subvariety inside the product of projective spaces 
$\bP(\Hom(\Lambda^{s_i}E_1,\Lambda^{s_i}E_2))$, $i=1,\dots,m$ defined as the closure
of points of the form $([\Lambda^{s_i}\varphi])_{i=1}^m$ for maximal rank maps
$\varphi\in \Hom(E_1,E_2)$. The complete case ($\bs=(1,\dots,N-1)$) is known 
to be smooth \cite{Vain84,M20,Th99}. We prove the following theorem (see precise statement in
Proposition \ref{prop:rfibers}).

\begin{theoremA}
The fiber of the projection map $\Bl_\br\to \Gr_d(V)$ over a subspace $U$ is isomorphic 
to a collineation variety depending on $\dim U\cap V_{n-d}$. The complete mixed blow-up $\Bl_{(2,\dots,d)}$ is smooth
and serves as a desingularization for all other $\Bl_\br$. 
\end{theoremA}

Before passing to a more general situation, let us mention two similar constructions found in 
the literature. First, in \cite{Fa26,FW25,FZ21} the authors consider the blow-ups of Grassmannians 
along zero schemes of Pl\"ucker coordinates of fixed degree.
They define the generalized Kausz compactifications \cite{Kaus00} and describe
the link to the spaces of collineations. Second, in \cite{FO24,FGNP26,Sv25} the authors study
truncated Grassmannians for certain allowed sets of degrees. Our setup is different, since
we want to keep the $\exp(\fa_d)$ (and $P_d^-$) symmetry.       

Finally, let us describe the general case. 
Let $G$ be a simple simply-connected Lie group with the   flag variety $G/B$
(the case of arbitrary $G/P$ is similar).
Let $\la$ be a regular dominant integral weight and let $L(\la)$ be the corresponding
irreducible highest weight $\fg$ module with the highest weight vector $v(\la)$. 
We fix the projective embedding $G/B\subset \bP(L(\la))$ with the image being the closure 
of the $B_-$ orbit through the highest weight line $[v(\la)]$.
Let $\eO(\la)$ be the line bundle on $G/B$ obtained as a pull back of $\eO(1)$ on 
$\bP(L(\la))$.
For a Weyl group element $\sigma$ let $v(\sigma\la)\in L(\la)$ be the corresponding 
extremal weight vector. The (opposite) Demazure module $D(\sigma\la)\subset L(\la)$
is the $\U(\fb_-)$ span of $v(\sigma\la)$. The (opposite) Schubert variety 
$S_\sigma$ sits inside $\bP(D(\sigma\la))$ as the closure of the $B_-$ orbit of 
the line $[v(\sigma\la)]$. 

We define the truncated $\fb_-$ module $L_\sigma(\la)$ as the quotient $L(\la)/D(\sigma\la)$
and denote the image of $v(\la)$ as $v_\sigma(\la)$; note that the truncated module
$L_\sigma(\la)$ is cyclic with cyclic vector $v_\sigma(\la)$. We also define the 
truncated flag variety $X_\sigma(\la)\subset \bP(L_\sigma(\la))$ as the closure 
of $B_-[v_\sigma(\la)]$.

\begin{theoremA}\label{thm:gen}
The blow-up of $G/B$ along the Schubert variety $S_\sigma$ admits a closed embedding
into $G/B\times X_\sigma(\la)$ as the closure of the $B_-$ orbit through the product
$[v(\la)]\times [v_\sigma(\la)]$. 
\end{theoremA} 
 
As in the case of Grassmann varieties, it is natural to consider the mixed blow-ups.
More precisely, one starts with a collection of the Weyl group elements
$\underline{\sigma}=(\sigma_1,\dots,\sigma_m)$ and define $F_{\underline{\sigma}}$
inside $\prod_{i=1}^m X_{\sigma_i}(\la)$ as the closure of the $B_-$ orbit through the
product of highest weight lines $[v_{\sigma_i}(\la)]$. It is tempting to conjecture           
that if $\underline{\sigma}=W$ (i.e. $m=|W|$), then  $F_{\underline{\sigma}}$ is smooth.

Our paper is organized in the following way. In Section \ref{sec:setup} we introduce
the main objects of study and collect the preliminary material.
In Section \ref{sec:rblow-up} we describe the blow-ups along the Schubert varieties
$S_r$ as orbit closures inside the products of classical and truncated Grassmannians.
In Section \ref{sec:resol} we construct a resolution for the ideal sheaf of $S_r$ and derive the general
description of the blow-ups in this particular case. 
In Section \ref{sec:coll} we introduce the mixed blow-ups and describe the fibers 
of the projections to the Grassmannians in terms of partial collineations.
In Section \ref{sec:general} we provide the general setup describing the blow-ups
along Schubert varieties in terms of the truncated flag varieties.
Appendix \ref{sec:app} contains technical computations with the ideal sheaves.

\subsection*{Conventions}
Throughout the paper we work over the field of complex numbers. We fix two positive 
integers $d<n$ and always assume that $d\le n-d$. 
For a vector $u$ in a vector space $U$ we denote by $[u]\in\bP(U)$ the line spanned
by $u$.  
All the Demazure modules and Schubert
varieties we consider are the opposite Demazure modules and opposite Schubert varieties. 
Hence, the Demazure modules are generated from the extremal weight vectors by the 
action of the Chevalley generators corresponding to the negative roots and the Schubert
varieties are invariant with respect to the negative Borel subgroup.  

\section{The setup}\label{sec:setup}
\subsection{Lie algebras and groups}
Throughout the text we use the notation $[k]$ for 
the set $\{1,\dots,k\}$, $k\in\bZ_{>0}$. We fix two numbers $d,n\in\bZ_{\ge 1}$ such 
that $1\le d<n$. We assume that $d\le n-d$ to ensure that $\min(d,n-d)=d$
(this assumption is not conceptual, but makes the notation simpler).

Let $V$ be an $n$-dimensional vector space with a fixed basis $v_1,\dots,v_n$.
Let $\msl_n=\msl(V)$ be the Lie algebra of complex traceless matrices with the standard
Cartan decomposition $\msl_n=\fn_-\oplus\fh\oplus \fn$ into the lower triangular,
diagonal and upper triangular subalgebras. 
In particular, $\fn_-$ is spanned by matrix units $E_{i,j}$ with $1\le j<i\le n$.
The corresponding subgroups of the Lie group $SL_n$ are denoted by 
$N_-$, $H$ and $N$. Also, let $\fb=\fn\oplus\fh$ and $\fb_-=\fn_-\oplus\fh$
be the Borel subalgebras with the Lie groups $B,B_-\subset SL_n$.

For a set $J\subset [n-1]$ we denote by $\fp_J\subset\msl_n$ the corresponding
standard parabolic subalgebra containing $\fb$. In particular, if $J=[n~-~1]$, 
then $\fp_J=\fb$ and if $J=\{d\}$ is a single element, then $\fp_d$
is the maximal parabolic subalgebra corresponding to the simple root $\al_d$.
We denote by $P_J\subset SL_n$ the parabolic Lie subgroup corresponding to the Lie
algebra $\fp_J$. Similarly, one has the opposite parabolic subagebras $\fp^-_{J}\supset\fb_-$
and their Lie groups $P^-_{J}\subset SL_n$. 

For $d=1,\dots,n-1$, let $\fa_d\subset \fn_-$ be the lower triangular abelian radical of the  
maximal parabolic subalgebra $\fp^-_d$, i.e. $\msl_n=\fa_d\oplus \fp_d$.
Explicitly, the radical, $\fa_d$ is spanned by the matrix units $E_{i,j}$ with 
$i=d+1,\dots,n$ and $j=1,\dots,d$. Let $\exp(\fa_d)\subset N_-$ be the corresponding 
abelian unipotent group, $\exp(\fa_d)\simeq \bG_a^{d(n-d)}$ ($\bG_a=(\bC,+))$.

\subsection{Representations}
Let $\al_1,\dots,\al_{n-1}\subset\fh^*$ and $\omega_1,\dots,\omega_{n-1}\subset\fh^*$ 
be the simple (positive) roots and fundamental weights for $\msl(V)$. For a dominant 
integral weight
$\la\in P^+$, $\la=\sum_{i=1}^{n-1} m_i\omega_i$, $m_i\in\bZ_{\ge 0}$ we denote by
$\Sigma_\la V$ the corresponding irreducible finite-dimensional $\msl(V)$ module with a fixed 
highest weight vector $v(\la)$.
In particular, $\Sigma_{\omega_1}V$ is the vector representation $V$ and 
$\Sigma_{\omega_d}V\simeq \Lambda^d V$.

One can also use the $\mgl(V)\simeq \mgl_n$ terminology. Namely, the irreducible finite-dimensional 
$\mgl_n$ modules $\Sigma_\mu V$ are labeled by partitions $\mu=(\mu_1\ge\dots\ge\mu_n\ge 0)$. 
The restriction of an irreducible $\mgl_n$ module to the subalgebra $\msl_n\subset\mgl_n$
is still irreducible and is given by
\[
\Sigma_\mu V|_{\msl_n} \simeq \Sigma_{(\mu_1-\mu_2)\om_1 + \dots + (\mu_{n-1}-\mu_n)\om_{n-1}} V.
\]
In particular, the restriction of  $\Sigma_{1^d0^{n-d}}V$
to $\msl(V)$ is the fundamental representation $\Sigma_{\om_d} V$ (here and below  
$a^b=(a,\dots,a)$, where $a$ shows up $b$ times).

\begin{rem}
In what follows we use simultaneously representations of algebras $\msl(V)$ and $\mgl(V)$
for different spaces $V$; hence it is convenient for us to specify the underlying space $V$
explicitly while working with the Lie algebras and their highest weight representations. 
\end{rem}

\begin{rem}
 The $\mgl(V)$ modules $\Sigma_\mu V$ for $\mu_1=\dots =\mu_n$ are one-dimensional and restrict
 to the trivial $\msl(V)$ module $\Sigma_0 V$.
\end{rem}

Let $v_1,\dots,v_n$ be the standard basis of $V=\Sigma_{\om_1} V$.
For $1\le d<n$ we fix the decomposition
\[
V=V_d\oplus V_{n-d},\ V_d = \mathrm{span}\{v_1,\dots,v_d\},
V_{n-d} = \mathrm{span}\{v_{d+1},\dots,v_n\}.
\]
By definition, $\fa_d V_d \subset V_{n-d}$ and both $V_d$ and $V_{n-d}$ are 
preserved by the Levi subalgebra $\fl_d=\msl_d\oplus\msl_{n-d}\subset \fp_d$.

The space $\Sigma_{\omega_d}(V)=\Lambda^d V$ is a cyclic representation of $\fa_d$ with the cyclic 
vector being the highest weight vector $v(\omega_d)=v_1\wedge\dots\wedge v_d$.
For a set $I\subset[n]$, $I=\{i_1<\dots<i_d\}$ we denote by $v_I$ the 
wedge product $v_{i_1}\wedge\dots\wedge v_{i_d}\in \Lambda^dV$. In particular,
$v(\omega_d)=v_{[d]}$. We introduce the degree of $I$ (and of the vector $v_I$)
by 
\[
\deg I = |I_{>d}| = \#\{i\in I:\ i>d\}.
\] 
The degree ranges from $0$ (for $I=[d]$) to $d$ (recall the assumption $d\le n-d$).

\begin{lem}
The wedge power $\Lambda^d V$ is a cyclic module of the parabolic subalgebra 
$\fp^-_{d}$. The Levi subalgebra $\fl_d$ acts by the degree preserving operators
and $\fa_d$ increases the degree by one.
\end{lem}
\begin{proof}
The module $\Lambda^d V$ is cyclic since $\fa_d$ generates the whole space from 
the highest weight vector $v(\omega_d)=v_{[d]}$. The second claim is obvious.  
\end{proof}

An opposite Demazure module inside an irreducible $\msl(V)$ module is a subspace generated
from an extremal weight vector by the action of the Borel subalgebra. For the wedge 
powers $\Lambda^d V$ the extremal vectors are exactly vectors $v_I$, $I\in\binom{[n]}{d} $.

\subsection{Cyclic quotients}\label{subsec:cycquot}
For an integer $r\ge 0$ we 
introduce a family $(\Lambda^d V)_r$ of cyclic $\fa_d$ quotients of $\Lambda^d V$
as follows 
\[
(\Lambda^d V)_r = \Lambda^d V/\mathrm{span}(v_I,\ \deg I > r).
\]  
In particular, $\dim (\Lambda^d V)_0 =1$, $\dim (\Lambda^d V)_1 =1+d(n-d)$,
$(\Lambda^d V)_d =\Lambda^d V$. 
We denote by $w_r\in (\Lambda^d V)_r$ the class of the highest weight vector 
$v(\om_d)$.
One has $(\Lambda^d V)_r = \U(\fa_d)w_r$ and the degree grading on $(\Lambda^d V)_r$
is compatible with the action of $\fa_d$ by the degree one operators.

\begin{lem}
The space $\mathrm{span}(v_I,\ \deg v_I>r)$ is $\fp^-_{d}$-invariant.
One has the isomorphism of $\fl_d = \msl_d\oplus\msl_{n-d}$ modules
\begin{equation}\label{eq:Hom}
(\Lambda^d V)_r \simeq \bigoplus_{i=0}^{r} \Lambda^{d-i} V_d\T \Lambda^i V_{n-d}
\simeq \bigoplus_{i=0}^{r} \Hom(\Lambda^i V_d,\Lambda^i V_{n-d}).
\end{equation}
\end{lem}
\begin{proof}
The first claim follows from $\fa_d V_d\subset V_{n-d}$. To prove the second claim we note
that $(\Lambda^d V)_r$ admits a basis of the form $v_J\wedge v_L$, where 
\[
J=(1\le j_1\le\dots\le j_{d-i}\le d),\ L=(d+1\le l_1<\dots<l_i\le n)
\] 
for $i=0,\dots,d$, which implies the first isomorphism in \eqref{eq:Hom}. To prove the 
second isomorphism it suffices to note that $\Lambda^{d-i} V_d\simeq (\Lambda^i V_d)^*$.
\end{proof}

\begin{cor}
One has 	
\[
\dim (\Lambda^d V)_r = \sum_{i=0}^{r} \binom{d}{i}\binom{n-d}{i}.
\]
\end{cor}

\begin{rem}
The truncated spaces $(\Lambda^d V)_r $ are naturally $\fp^-_{d}$ modules, but they also admit 
the action of the larger algebra $\msl^{(d)}_n$, which is a degeneration of the classical $\msl_n$
(see \cite{Fe23,BR24,PY13}). 	
The Lie algebra $\msl^{(d)}_n\simeq \fp_d\oplus\fa_d$ is isomorphic to $\msl_n$ as a vector space, 
$\fp_d$ is a subalgebras, $\fa_d$ is an 
abelian ideal and the action of $\fp_d$ on $\fa_d$ comes from the 
quotient realization $\fa_d\simeq \msl_n/\fp_d$. The existence of this larger algebra of symmetries 
is a powerful tool in many problems (see e.g. \cite{FFL11,Fe25}), but in this paper we never
use it.
\end{rem}

For $0\le r\le d$ let 
\begin{equation}\label{eq:I(r)}
	I(r)=(1,\dots,d-r,d+1,\dots,d+r);
\end{equation}
in particular, $I(0)=[d]$, $I(d)=\{d+1,\dots,2d\}$ (recall the assumption $2d\le n$).

\begin{lem}\label{lem:Demmod}
The space $\mathrm{span}(v_I,\ \deg I>r)\subset \Lambda^d V$ is an 
opposite Demazure module  $\U(\fn_-) v_{I(r+1)}$.
\end{lem}
\begin{proof}
One easily sees that if $\deg v_I>r$, then $v_I$ can be obtained from the pure wedge
$v_{I(r+1)}$ by applying several matrix units of the 
form $E_{i,j}$, $i>j$.
\end{proof}

\begin{rem}
Any $\fp^-_d$ invariant Demazure submodule of $\Lambda^dV$ is equal to the span
of vectors $v_I$, $\deg v_I>r$ for some $r$. In fact, let us fix a collection 
$J\in\binom{[n]}{d}$. Then $\U(\fl_d)v_J$ contains the vector  
$v_{I(r+1)}$, where $r+1=\#\{j\in J:\ j>d\}$. 
\end{rem}

\subsection{Grassmannians and Schubert varieties}
Let $\Gr_d(V)$ be the Grassmannian of $d$-dimensional subspaces of $V$. 
The Pl\"ucker embedding realizes $\Gr_d(V)$ as a projective algebraic subvariety
of $\bP(\Lambda^dV)$ of dimension $d(n-d)$.  
The Grassmannians admit the transitive action of $SL_n$ and $\Gr_d(V)\simeq SL_n/P_d$. 
The action of the abelian unipotent subgroup $\exp(\fa_d)$ is not transitive, but 
admits an open dense orbit through the highest weight line containing $v_{[d]}$;
this orbit is an affine cell $\bA^{d(n-d)}$. 

More generally, for a collection 
$I\in \binom{[n]}{d}$ the open (opposite) Schubert variety $S^\circ_I$ is defined as the
$B_-$ orbit through the line $[v_I]\in\bP(\Lambda^dV)$. This orbit is an affine cell and
its closure is the (opposite) closed Schubert variety. Hence, the Grassmannian is the disjoint
union of the affine cells -- the open Schubert varieties. We denote by $S_I$ the closure 
of $S^\circ_I$ and by $p_I\in S_I\subset \Gr_d(V)$ the point corresponding to $[v_I]$. Explicitly,
$p_I$ is spanned by the vectors $v_i$ with $i\in I$. 

In what follows we use the special family $S_r=S_{I(r)}$, $r=0,\dots, d$ 
of opposite Schubert varieties.  

\begin{lem}\label{lem:dimint}
A $d$-dimensional subspace $U\subset V$ belongs to $S_r$ if and only if 
$\dim U\cap V_{n-d}\ge r$.	
\end{lem}
\begin{proof}
For the torus fixed point $p_{I(r)}$ one has $\dim p_{I(r)}\cap V_{n-d}= r$. Since
$B_-$ preserves $V_{n-d}$, we conclude that $\dim U\cap V_{n-d}\ge r$ for any $U\in S_r$.

Now assume $\dim U\cap V_{n-d}\ge r$. Then $U$ belongs to a Schubert cell $S_J^\circ$ 
for some $J$ subject to the condition $|\{j\in J:\ j>d\}|\ge r$. Hence $I(r)\ge J$
(componentwise, being ordered from smaller to larger entries) and $S_{I(r)}\supset S_J$.	
\end{proof}

One has the following list of basic properties of the Schubert varieties $S_r$.
\begin{lem}
The varieties $S_r$, $0\le r\le d$ satisfy the following properties:	
\begin{enumerate}
\item $S_0\supset S_1\supset \dots \supset S_d$,
\item $S_0=\Gr_d(V)$, $S_d =\Gr_d(V_{n-d})$,
\item $S_1$ is the divisor in $\Gr_d(V)$, $\Gr_d(V)\setminus S_1 = \exp(\fa_d)[v_{[d]}]$,
\item $\mathrm{codim} S_r = r^2$.
\end{enumerate}
\end{lem}

\begin{rem}
The Schubert varieties inside Grassmannian $\Gr_d(V)$ are often parametrized by partitions 
$\la=(\la_1,\dots,\la_k)$ with $n-d\ge \la_1$ and $k\le d$ (i.e. the corresponding Young
diagram fits into the $d\times (n-d)$ box). In particular, the dimension of the Schubert variety
attached to $\la$ is the sum of all $\la_i$. 
In this parametrization the Schubert variety $S_r$ corresponds to the partition 
$(n-d)^{d-r}(n-d-r)^r$. Here is a picture for $n=10, d=4$ and $r=1,2,3,4$.
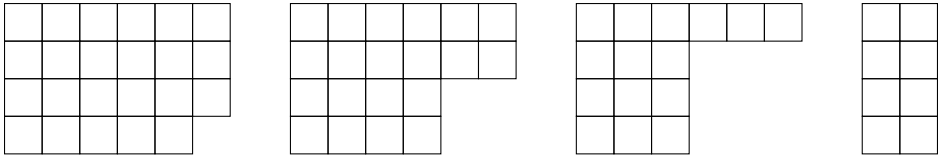
\[
	\begin{tikzpicture}[scale=0.5]
		\foreach \row[count=\y] in {6,6,6,5}{
			\foreach \x in {1,...,\row}{
				\draw (\x-1,-\y+1) rectangle (\x,-\y+2);
			}
		}
	\end{tikzpicture} \qquad
	\begin{tikzpicture}[scale=0.5]
		\foreach \row[count=\y] in {6,6,4,4}{
			\foreach \x in {1,...,\row}{
				\draw (\x-1,-\y+1) rectangle (\x,-\y+2);
			}
		}
	\end{tikzpicture} \qquad
	\begin{tikzpicture}[scale=0.5]
		\foreach \row[count=\y] in {6,3,3,3}{
			\foreach \x in {1,...,\row}{
				\draw (\x-1,-\y+1) rectangle (\x,-\y+2);
			}
		}
	\end{tikzpicture} \qquad
	\begin{tikzpicture}[scale=0.5]
		\foreach \row[count=\y] in {2,2,2,2}{
			\foreach \x in {1,...,\row}{
				\draw (\x-1,-\y+1) rectangle (\x,-\y+2);
			}
		}
	\end{tikzpicture} 
\]
\end{rem}

\begin{rem}
The previous remark implies that the varieties $S_r$ are singular for $r\ne 0,d$, since
the smooth Schubert varieties correspond to rectangular partitions \cite{LS84}.	
\end{rem}

\begin{lem}
A Schubert variety in $\Gr_d(V)$ is $P^-_{d}$ invariant if and only if it coincides
with one of $S_r$. The variety $S_r$ consists of points of $\Gr_d(V)$ contained
in the projectivization of the Demazure module $\U(\fn_-)v_{I(r)}$.	
\end{lem}
\begin{proof}
The maximal parabolic subgroup $P^-_{d}$ preserves $V_{n-d}$ and hence by Lemma \ref{lem:dimint}
all $S_r$ are $P^-_{d}$ invariant. Now let $S_I$ be a $P^-_{d}$ invariant Schubert variety.
Let $I=I_{<}\sqcup I_{>}$, where $I_< = I\cap [d]$, $I_> = \{i\in I:\ i>d\}$, and let
$r=|I_>|$. Then there exists an element $g\in P^-_{d}$ such that $gp_I=p_{I(r)}$ for 
$I(r)$ defined by \eqref{eq:I(r)}. Hence $S_I=S_{I(r)}$, because $S_r$ is $P^-_{d}$
invariant. The last claim is the standard link between a Schubert variety and the corresponding 
Demazure module (see e.g. \cite{Kum02}), but can be easily proved explicitly in this case.
\end{proof}

\subsection{Vector bundles on Grassmannians}\label{subsec:Schurf}
Let us consider the tautological vector bundle $\eU_d$ on $\Gr_d(V)$, i.e. $\eU_d$
is a rank $d$ subbundle of $V(\eO) = V\T \eO_{\Gr_d(V)}$, whose fiber at a point $U$ 
is $U$ itself. We denote by $\eU_d^\perp$ the vector bundle
$(V(\eO)/\eU_d)^*$. Let $\eO(1)$ be the line bundle on $\Gr_d(V)$, which is the positive 
generator of the
Picard group of $\Gr_d(V)$. In particular,  $\eO(1)$ is the inverse image of $\eO_{\bP(\Lambda^dV)}(1)$
with respect to the standard Pl\"ucker embedding. 

For a collection of integers $\mu=(\mu_1\ge \dots \ge \mu_k)$ and a rank $k$ vector
bundle $\eL$, let $\Sigma_\mu \eL$ be the Schur functor applied to $\eL$ (see e.g. \cite{W03}).
In particular,
\[
\Sigma_{(1,\dots,1)}\eU_d^* \simeq \Sigma_{(-1,\dots,-1)}\eU_d^\perp\simeq \eO(1). 
\] 
We note also that for any vector bundle $\eL$ and a composition $\mu$ one has 
\begin{equation}\label{eq:dualbun}
\Sigma_{(\mu_1,\dots,\mu_k)}\eL \simeq \Sigma_{(-\mu_k,\dots,-\mu_1)}\eL^*. 
\end{equation}
In what follows we denote $(-\mu_k,\dots,-\mu_1)$ by $\mu^*$.

For two collections $\mu=(\mu_1\ge \dots \ge \mu_d)\in\bZ^d$ and 
$\nu=(\nu_1\ge \dots\ge \nu_{n-d})\in\bZ^{n-d}$,
the cohomologies of the vector bundles $\Sigma_\mu \eU_d\T \Sigma_\nu \eU_d^\perp$ are computed 
via the Borel-Weil-Bott theorem as follows (see \cite{D76,Kap84,Kap88}). Let $\rho=(n,n-1,\dots,1)$ be 
the half sum of positive roots
and let us denote by $\ell(\sigma)$ the length of a permutation $\sigma$. Also, let 
$\kappa=(\mu,\nu)\in\bZ^n$. Then all the cohomologies of  $\Sigma_\mu \eU_d\T \Sigma_\nu \eU_d^\perp$
vanish unless $\kappa+\rho$ has no coinciding entries. If this condition holds, then
\begin{equation}\label{eq:BWB}
H^k(\Gr_d(V),\Sigma_\mu \eU_d\T \Sigma_\nu \eU_d^\perp) = 
\begin{cases}
\Sigma_{\sigma(\kappa+\rho)-\rho} V^*, & k=\ell(\sigma),\\
0, & \text{ otherwise},
\end{cases}
\end{equation}
where $\sigma(\kappa+\rho)$ is strictly decreasing.

\subsection{Truncated Grassmannians and their relatives}\label{sec:truncGr}
Recall the cyclic vector $w_r\in (\Lambda^dV)_r$ -- the image of the highest weight vector
$v_{[d]}\in \Lambda^dV$.  
\begin{dfn}\label{dfn:trGr}
For $r=1,\dots,d$ we define 
\[
X_r = \overline{\exp(\fa_d).[w_r]} \subset \bP((\Lambda^d V)_r).
\]	
\end{dfn}

So, $X_r$ is the closure of the orbit of the abelian unipotent group $\bG_a^{d(n-d)}$.  

\begin{lem}
All varieties $X_r$  are birationally isomorphic and admit the action of the parabolic
group $P_d^-$. The boundary members of the family are
$X_1=\bP^{d(n-d)}$ and $X_d=\Gr_d(V)$.
\end{lem}
\begin{proof}
All the varieties $X_r$ share the same open part $\exp(\fa_d).[w_r]$
isomorphic to $\bA^{d(n-d)}$. The truncated spaces $(\Lambda^d V)_r$ admit the action of
$P_d^-\supset \exp(\fa_d)$ and $[w_r]$ is invariant with respect ot the Levi part
$SL(V_d)\times SL(V_{n-d})$. Hence the $\exp(\fa_d)$ action on $X_r$ extends to the 
action of $P_d^-$. 

For $r=1$ the space $(\Lambda^dV)_1$ is spanned by $w_1$ and
the images of vectors of the form $v_I$ with $I=[d]\setminus\{j\}\cup\{i\}$ for some 
$1\le j\le d$, $d+1\le i\le n$. Hence, we identify $(\Lambda^dV)_1$ with $\bK w_1\oplus \fa_d$.
The $\fa_d$ action on $(\Lambda^dV)_1$ is trivial on the second summand and the action
on $w_1$ is induced 
by the isomorphism $\fa_d\T \bK w_1\simeq \fa_d$. Hence for $a\in\fa_d$ one has
$\exp(a)[w_1]=[w_1+a]$ and $X_1=\bP (\Lambda^dV)_1=\bP^{d(n-d)}$.
Finally, $(\Lambda^dV)_d=\Lambda^dV$ and hence $X_d=\Gr_d(V)$.
\end{proof}

\begin{rem}
Identifying $\fa_d$ with $\bA^{d(n-d)}$, one gets the maps $\bA^{d(n-d)}\to X_r$. In particular,
for $r=d$ one gets the so called  Landsberg-Manivel map $\bA^{d(n-d)}\to \Gr_d(V)$
(see \cite{LM02,FW25}).
\end{rem}

In \cite{BSS25,FSS25} the authors considered the closure 
of the graph 
of the birational map $\bP^{d(n-d)}\to \Gr_d(V)$, induced 
by the Landsberg-Manivel map. This graph closure $\G_d(V)$ can be described as 
\begin{equation}\label{eq:graphcl}
\G_d(V) = \overline{\exp(\fa_d).\left([w_1]\times [w_d]\right)}
\subset \bP^{d(n-d)}\times\Gr_d(V) = X_1\times X_d,
\end{equation}
where $[w_d]=[v_{[d]}]$ is the highest weight line in $\Gr_d(V)=X_d$. 
It is proved in \cite{Fe24} that all fibers of the natural projection $\pi:\G_d(V)\to \Gr_d(V)$
are projective spaces. By definition, the fibers are subvarieties of the projective space 
$\bP(\Lambda^d(V)_1)\simeq \bP(\bK w_1\oplus \Hom(V_d,V_{n-d}))$.  
The following is proved in \cite{Fe24}:
\begin{enumerate}
\item if $U\cap V_{n-d}= 0$, then $\pi^{-1}U$ is a single point,
\item $\pi^{-1}U\subset \bP(\Hom(V_d,V_{n-d}))$ for $U$ such that $U\cap V_{n-d}\ne 0$,
\item if $U\cap V_{n-d}\ne 0$, then 
\[
\pi^{-1}U \simeq \bP(\Hom(V_d/\mathrm{pr}(U),U\cap V_{n-d})),
\]
\end{enumerate}
where $\mathrm{pr}:V\to V_d$ is a projection along $V_{n-d}$. In particular,
if the dimension of $U\cap V_{n-d}$ is equal to $m$, then $\pi^{-1}U\simeq \bP^{m^2-1}$.

\begin{example}
Let $n=4$, $d=2$. Then the preimage of the point  $p_{\{3,4\}}$ is isomorphic to 
$\bP^3$. Outside of this point the map $\pi:\G_d(V)\to \Gr_d(V)$
is one-to-one.  
\end{example}

\begin{rem}
The varieties $X_r$ admit the following explicit Pl\"ucker type realization. Let
$z_{i,j}$ be auxiliary variables with $i\in [n]$, $j\in [d]$. For $J\in\binom{[n]}{d}$
we denote by $\Delta_J(z)$ the $d\times d$ minor of the matrix $Z=(z_{i,j})$ supported
on rows from $J$. Recall that the degree $\deg(J)$ is the cardinality of $J_{>d}$ and
let $N_r$ be the number of collections $J$ such that $\deg J\le r$.
Then $X_r$ is realized inside the projective space $\bP^{N_r-1}$ with coordinates labeled by 
$J$, $\deg J\le r$ as follows. For a collections of numbers $z_{i,j}$ we denote by 
$\Delta(z)\in \bP^{N_r-1}$ the point whose $J$-th coordinate is equal to $\Delta_J(z)$.
Then $X_r$ is the closure of  the set of points $\Delta(z)$ for all values of $z_{i,j}$
such that the $n\times d$ matrix $Z$ is of (maximal possible) rank $d$.
\end{rem}

\section{Truncated Grassmannians and blow-ups}\label{sec:rblow-up}
The goal of this section is to describe the truncated Grassmannians and blow-ups $\Bl_{S_r} \Gr_d(V)$ for the $P^-_d$
invariant Schubert varieties $S_r$.  To simplify the notation, in this section 
we denote the blow-up $\Bl_{S_r} \Gr_d(V)$ by $\Bl_r$.

\subsection{Truncated Grassmannians}
Recall the truncated wedge powers from subsection \ref{subsec:cycquot}
\[
(\Lambda^d V)_{r-1}\simeq \bigoplus_{i=0}^{r-1} \Hom(\Lambda^iV_d, \Lambda^iV_{n-d})
\]
and the truncated Grassmannians $X_{r-1}$ sitting inside $\bP((\Lambda^d V)_{r-1})$ 
(Definition \ref{dfn:trGr}).
In order to describe the geometric structure of $X_{r-1}$ we prepare the following lemma.
\begin{lem}
The birational projection map 
$\bP(\Lambda^dV)\dashrightarrow \bP((\Lambda^d V)_{r-1})$
restricted to $\Gr_d(V)$ is well defined embedding outside $S_{r-1}$. 
\end{lem}
\begin{proof}
We note that the complement $\Gr_d(V)\setminus S_{r-1}$ is the union of Schubert cells 
(open Schubert varieties)
$S^\circ_J$ with $J\in\binom{[n]}{d}$, $\deg J\le r-2$. The affine coordinates on $S^\circ_J$ 
are $a_{i,j}$ with 
$j\in J$, $i\notin J$ and $i>j$; for $J=(j_1<\dots <j_d)$ the Pl\"ucker coordinates $\Delta_I$ of a 
point $(a_{i,j})$ are the minors of the matrix of the map sending $v_s$, $s\in [d]$ to 
$v_{j_s}+\sum_{i>j} a_{i,j}v_i$.

Now let $(a_{i,j})$ be coordinates of $U\in \Gr_d(V)\setminus S_{r-1}$. Under the projection
$\bP(\Lambda^dV)\dashrightarrow \bP(\Lambda^d V)_{r-1}$ one forgets all the Pl\"ucker coordinates 
$\Delta_I(U)$ such that $\deg I\ge r$. We need to check that no information is lost 
under this projection for $\Gr_d(V)\setminus S_{r-1}$. To this end, we take a pair 
$i>j$, $j\in J$, $i\notin J$ and 
consider the collection $J\setminus \{j\}\cup \{i\}$. Then 
\[
\Delta_{J\setminus \{j\}\cup \{i\}}(U)=\pm a_{i,j}\ \text{ and }\
\deg J\setminus \{j\}\cup \{i\}\le r-1.  
 \] 
Hence $\Gr_d(V)\setminus S_{r-1}$ embeds into $\bP((\Lambda^d V)_{r-1})$.
\end{proof}
 
\begin{cor}
The truncated Grassmannian $X_{r-1}$ contains $\Gr_d(V)\setminus S_{r-1}$.
\end{cor}

\begin{example}
Let $r=2$ (the $r=1$ case is trivial, since $S_0=\Gr_d(V)$). Then Corollary above says that
$X_1\supset \Gr_d(V)\setminus S_{1}$, which is the open Schubert cell in the Grassmannian 
(the $\exp(\fa_d)$ orbit of the highest weight line). One easily sees that $X_1$ is simply
the projective space $\bP(\Lambda^d(V)_1)\simeq \bP^{d(n-d)}$.
\end{example}

Now let us describe the complement  of $\Gr_d(V)\setminus S_{r-1}$ in $X_{r-1}$. 
Let $R_{r-1;d,n-d}$ be the algebra generated by all $(r-1)$ minors of a $d\times (n-d)$ matrix
inside a polynomial ring in variables $z_{i,j}$, $d+1\le i\le n$, $1\le j\le d$ \cite{BV88,BCV13}. 
Then $\Proj(R_{r-1;d,n-d})$ is 
isomorphic to the closure in $\bP\Hom(\Lambda^{r-1}V_d,\Lambda^{r-1}V_{n-d})$ of the
set of elements of the form $[\Lambda^{r-1}\varphi]$, $\varphi \in\Hom(V_d,V_{n-d})$.
By definition, we have the embedding $\Proj(R_{r-1;d,n-d})\subset \bP(\Lambda^dV)_{r-1}$
(since $(\Lambda^dV)_{r-1}\supset \Hom(\Lambda^{r-1}V_d,\Lambda^{r-1}V_{n-d})$).

\begin{prop}
$X_{r-1}$ is the union of $\Gr_d(V)\setminus S_{r-1}$ and $\Proj(R_{r-1;d,n-d})$.
\end{prop}
\begin{proof}
The complement of $\Gr_d(V)\setminus S_{r-1}$ in $X_{r-1}$ is equal to the intersection
$X_{r-1}\cap \bP\Hom(\Lambda^{r-1}V_d,\Lambda^{r-1}V_{n-d})$. We claim that this 
intersection coincides with $\Proj(R_{r-1;d,n-d})$. In fact, by definition, $X_{r-1}$ is
the closure inside $\bP((\Lambda^dV)_{r-1})$ of the open $\exp(\fa_d)$ orbit of $[w_{r-1}]$, 
which is naturally identified with $\fa_d\simeq \Hom(V_d,V_{n-d})$. 
Let us denote the matrix of a homomorphism $\varphi\in \Hom(V_d,V_{n-d})$ by $z_{i,j}$,
$i=d+1,\dots,n$, $j=1,\dots,d$. 
The $I$-th Pl\"ucker coordinate of the point corresponding
to $\varphi$ is the minor $\Delta_I$ of the submatrix of the $n\times d$ matrix 
$z_0\mathrm{Id}_{d\times d} + \varphi$ supported on rows from $I$ (here $z_0$ is an auxiliary
variable). Hence, $\Delta_I$ is a polynomial in $z_{i,j}$ of total degree $\deg I$ multiplied 
by $z_0^{d-\deg I}$; note that possible values of $\deg I$ for $(\Lambda^dV)_{r-1}$ are from 
$0$ to $r-1$.

Now assume we have a polynomial in Pl\"ucker variables vanishing at all points of the open orbit.
We are interested in all the solutions such that $\Delta_I=0$ for $\deg I<r-1$.    
We claim that these are exactly the zeroes of all the polynomials involving only the maximal
degree terms (i.e. only $\Delta_I$ with $\deg I = r-1$). In fact, all the relations for the
points of the orbit are degree-homogeneous. Hence, if a monomial of a relation involves  $\Delta_I$ 
with $\deg I < r-1$ as a factor, then any other monomial should also have a factor 
$\Delta_{I'}$  with $\deg I' < r-1$. Since we are looking at the zeroes of the 
Pl\"ucker coordinates of degree less than $r-1$, such a relation  automatically holds true.

Finally, let us look at the coordinate $\Delta_I$ with $\deg I = r-1$. 
There exist $\al=(1\le \al_1<\dots < \al_{r-1}\le d)$ and 
$\beta=(d+1\le \beta_1<\dots < \beta_{r-1}\le n)$, such that  $\Delta_I(\varphi)$ is the
$(r-1)$ minor of the matrix $\varphi$ supported on rows $\al$ and columns $\beta$ (multiplied
by $z_0^{d-r+1}$). Hence the relations between such $\Delta_I$ are exactly the relations in
$R_{r-1;d,n-d}$.
\end{proof}

\begin{rem}
The scheme $\Proj(R_{r-1;d,n-d})$ is isomorphic to the the closure of the image of 
the natural birational map between the projective homomorphism spaces 
$\bP\Hom(V_d,V_{n-d})\dashrightarrow \bP\Hom(\Lambda^{r-1}V_d,\Lambda^{r-1}V_{n-d})$ sending
$[\varphi]$ to $[\Lambda^{r-1}\varphi]$. These schemes are normal, Cohen-Macaulay and have
 rational 	singularities (see \cite{BV88,BCRV22}).
\end{rem}

\subsection{Fibers of the blow-ups}
Let $\pi_r: \Bl_r \to \Gr_d(V)$
be the projection map; our goal here is to describe the fibers $\pi_r^{-1}(U)$.
Let $J\in \binom{[n]}{d}$ be a cardinality $d$ subset and let $p_J\in\Gr_d(V)$ be the 
corresponding point in the Grassmannian.
Let $A_J$ be the standard affine cell containing $p_J$, i.e. 
\[
A_J\simeq \Hom(p_J,V/p_J) \simeq \Hom(\mathrm{span}(v_j, j\in J),\mathrm{span}(v_j, j\notin J)). 
\]
To a homomorphism $\varphi\in \Hom(p_J,V/p_J)$ we attach 
the subspace $U_\varphi=\mathrm{span}(v_j+\varphi v_j, j\in J)\in \Gr_d(V)$; in particular, 
for $\varphi=0$ one gets $U_0=p_J$.
 
We want to describe the intersection $A_J\cap S_r$. Let $m=\deg J$
be the cardinality of $J_{>d}$. Let $X_\varphi$ be the $m\times m$ submatrix of $\varphi$
supported on columns from $J_{>d}$ and rows from $[d]\setminus J_{\le d}$.

\begin{lem}
A point $U_\varphi$ belongs to the intersection $A_J\cap S_r$ if and only if 
the rank of $X_\varphi$ is at most $m-r$.
\end{lem}
\begin{proof}
The condition $\mathrm{rk} X_\varphi\le m-r$ is equivalent to 
\[
\dim \bigl(\mathrm{span}\{v_j+\varphi(v_j),\ j\in J_{>d}\}\cap V_{n-d}\bigr)\ge r,
\]
which in turn is equivalent to $\dim U_\varphi\cap V_{n-d}\ge r$.
\end{proof}

Let $Y_{N,c}\subset \mathrm{Mat}_{N\times N}$ be the determinantal variety consisting 
of matrices of rank at most $c$ \cite{Eis95}. 
The variety $Y_{N,c}$ is the vanishing set of the prime ideal generated by all $(c+1)\times (c+1)$
minors of an $N\times N$ matrix. 

\begin{cor}
Assume $m=|J_{>d}|\ge r$. Then  the	
intersection $A_J\cap S_r$ is isomorphic to a product of a determinantal variety and 
an affine space. 
\end{cor}
\begin{proof}
The condition $\mathrm{rk} X_\varphi\le m-r$ is equivalents to the vanishing of all 
$m-r+1$ minors of $X_\varphi$. Since there are no other conditions on the
entries of the matrix of $\varphi$, we arrive at the following isomorphism:
\[
A_J\cap S_r \simeq Y_{m,m-r}\times \bA^{d(n-d)-m^2}
\] 
(here $d(n-d)-m^2$ is equal to the number of entries of the matrix $\varphi$ which are not 
in $X_\varphi$).
\end{proof}

\begin{cor}
Here are several extreme cases:
\begin{itemize}
\item if $r=1$, then $A_J\cap S_r$ is cut out by a single equation $\det X_\varphi =0$;
\item if $r=2$, then $A_J\cap S_r$ is the common vanishing set of all maximal minors of $X_\varphi$;
\item if $r=m$, then $A_J\cap S_r$ is just the affine space $\bA^{d(n-d)-r^2}$;
\item if $r>m$, then $A_J\cap S_r$ is empty.
\end{itemize}
\end{cor}

Now let us consider the entries of the matrix $X_\varphi$ as variables.  
Let $R_{m,r}$ be the subring of the polynomial ring in entries of $X_\varphi$ generated 
by all $m-r+1$ minors $\Delta_{\alpha,\beta}$ 
for all 
$\al,\be\in\binom{m}{m-r+1}$. The projective spectrum of the ring $R_{m,r}$ is isomorphic to
the space of partial collineations. Namely, let 
\begin{equation}\label{eq:E1E2}
E_1=\mathrm{span}\{v_j:\ j\in J_{>d}\},\ 
E_2=\mathrm{span}\{v_j:\ j\in [d]\setminus J_{\le d}\}
\end{equation}
(in particular, $X_\varphi\in\Hom(E_1,E_2)$).
We consider the birational map
\[
\bP\Hom(E_1,E_2)\to \bP\Hom(\Lambda^{m-r+1} E_1,\Lambda^{m-r+1} E_2),\ f\mapsto \Lambda^{m-r+1} f
\]
and define partial collineations $\Col_{m-r+1}(E_1,E_2)$ as the closure of the image of this map.
We will also denote this space by $\Col_{m-r+1}(m)$, where $m=\dim E_1 =\dim E_2$.

\begin{lem}
One has $\Col_{m-r+1}(m)\simeq \Proj(R_{m,r})$. 
\end{lem}
\begin{proof}
By definition, the variety $\Col_{m-r+1}(m)$ sits in the projective space 
$\bP(\Lambda^{m-r+1} E_1,\Lambda^{m-r+1} E_2)$
with the coordinates labeled by pairs $\al,\be\in \binom{[m]}{m-r+1}$.
The defining relations cutting out $\Col_{m-r+1}(m)$ are exactly the (homogeneous) algebraic 
relations between the minors $\Delta_{\al,\be}$. 	
\end{proof}
 
Now let $\pi_r$ be the standard projection from the blow-up of $A_J\subset \Gr_d(V)$ along $A_J\cap S_r$ to $A_J$.  
We arrive at the following proposition, where $E_1,E_2$ are defined in \eqref{eq:E1E2}. 
\begin{prop}\label{prop:r-fibers}
Let $m=|J_{>d}|$. Then $\pi_r^{-1}(p_J)\simeq \Col_{m+1-r}(E_1,E_2)$.
\end{prop}
\begin{proof}
By definition, the blow-up $\Bl_{A_J\cap S_r}A_J$ sits inside the product of 
$A_J$ and $\bP\Hom(\Lambda^{m-r+1}E_1,\Lambda^{m-r+1}E_2)$ as the closure of the 
set of points of the form $(\varphi,[\Lambda^{m-r+1}X_\varphi])$.
Now $p_J\in A_J$ corresponds to $\varphi=0$ and hence the fiber over $p_J$ is exactly $\Proj(R_{m,r})\simeq \Col_{m-r+1}(m)$, which is 
realized as  $\Col_{m-r+1}(E_1,E_2)$.
\end{proof}

Let $\mathrm{pr}: V\to V_d$ be the projection sending $V_{n-d}$ to zero.

\begin{cor}
For a point $U\in \Gr_d(V)$ such that $\dim U\cap V_{n-d} = m$, one has 
\[
\pi_r^{-1} U \simeq \Col_{m-r+1}(U\cap V_{n-d}, V_d/\mathrm{pr} (U))\simeq 
\Col_{m-r+1}(m).
\] 	
\end{cor}
\begin{proof}
Follows from Proposition \ref{prop:r-fibers} by shifting the center of the cell $A_J$ from $p_J$ to $U$.	
\end{proof}

\begin{example}
Let $r=2$. Then all the fibers of the blow-up $\mathrm{Bl}_{S_2}\Gr_d(V)$ are projective spaces.
In fact, the algebra $R_{m,2}$ is generated by all the maximal minors (of size $(m-1)$), 
which are algebraically independent. So the fiber over a point $U\in \Gr_d(V)$ is isomorphic to $\bP^{m^2-1}$, where 	$m=\dim U\cap V_{n-d}$.
\end{example}

\begin{rem}
We note that the defining relations between the generating minors $\Delta_{\al,\be}$ in $R_{m,r}$ 
are very complicated for $r>2$, in particular, not quadratic (see e.g. \cite{BV88,BCV13}).  
\end{rem}

\section{Blow-ups of Grassmannians and ideal sheaves}\label{sec:resol}
The goal of this section is to construct a resolution  of the ideal sheaf of $S_r$.
As a consequence we derive the embedding of the blow-up 
 $\Bl_{S_r}\Gr_d(V)$ into $\Gr_d(V)\times X_{r-1}$.
 The existence of such an embedding also follows from the general 
construction from Section \ref{sec:general}.
  
 We consider the desingularization \cite{BK05,D74}
 \[
 \widetilde S_r = \{(U_r,U_d)\in \Gr_r(V_{n-d})\times \Gr_d(V):\ U_r\subset U_d\}. 
 \]
Clearly, $\widetilde S_r$ is smooth, as fibration over $\Gr_r(V_{n-d})$ with smooth fibers, and admits 
a natural projection $q:\widetilde S_r\to S_r$ sending $(U_r,U_d)$ to $U_d$
(since $U_d\supset U_r\subset V_{n-d}$, one gets $\dim U_d\cap V_{n-d}\ge r$).
For a point $(U_r,U_d)\in\widetilde S_r$ one has the following sequence
\begin{equation}\label{eq:UrUd}
U_r\hookrightarrow  V_{n-d}\subset V \twoheadrightarrow V/U_d\to 0
\end{equation}
and hence $\widetilde S_r$ is identified with the scheme of zeroes of a section of the following vector
bundle on $\Gr_r(V_{n-d})\times \Gr_d(V)$. Let $\eU_r$ be the tautological $r$-dimensional
vector bundle on $\Gr_r(V_{n-d})$ and let $\eU_d$ be the tautological $d$-dimensional
vector bundle on $\Gr_d(V)$. Then \eqref{eq:UrUd} implies that $\widetilde S_r$ is
the scheme of zeroes of an element of $H^0(\eU_r^*\T V(\eO)/\eU_d)$.
We thus arrive at the following lemma 
(we use the standard notation $\eU^\perp_d=(V(\eO)/\eU_d)^*$).

\begin{lem}
The ideal sheaf $\eI_{\widetilde S_r}$ admits the following Koszul resolution
\begin{equation}\label{eq:Kosres}
\dots\to\Lambda^2(\eU_r\T \eU_d^\perp)\to \eU_r\T \eU_d^\perp\to \eI_{\widetilde S_r}\to 0.	
\end{equation}
\end{lem}

Our next goal is to compute the resolution of the sheaf $\eI_{S_r}$. To this end we 
use the decomposition
\begin{equation}\label{eq:Lambda}
\Lambda^k(\eU_r\T \eU_d^\perp) \simeq \bigoplus_\la \Sigma_\la \eU_r \T \Sigma_{\la^t} \eU_d^\perp,
\end{equation}
where the sum ranges over partitions 
\begin{equation}\label{eq:part}
\la=(\la_1\ge\dots\ge \la_r), \ \la_1+\dots +\la_r = k,\  n-d\ge\la_1
\end{equation} 
and $\Sigma$ is the Schur functor. Recall that $S_r\subset \Gr_d(V)$ is a Schubert variety and thus has rational singularities \cite{BK05,Kum02}. Hence 
we can use the desingularization $q:\widetilde S_r\to S_r$ in order to produce
the resolution for the sheaf $\eI_{S_r}$. 
More precisely, we compute the direct image of \eqref{eq:Kosres} with respect to the map $q_*$. 
To this end, we need to compute the cohomologies of the sheaves $\Sigma_\la\eU_r$ along the 
fibers of the map $q$.

For a tuple $\la=(\la_1,\dots,\la_r)$ we use the notation $\la^*=(-\la_r,\dots,-\la_1)$ and $|\la|=\sum_{i=1}^r \la_i$. 
We also denote by $\rho$ the collection 
$(n-d,\dots,1)$ and if $\la^*+\rho$ has no coinciding entries, we let
$\sigma\in S_{n-d}$ denote the permutation reshuffling $\la^*+\rho$
into a decreasing order (see subsection \ref{subsec:Schurf}). 

\begin{lem}\label{lem:complex}
One has the following complex of sheaves on $\Gr_d(V)$:
\[
0\to \eF_{r^2}\to \dots \to \eF_1\to \eF_0\to\eO_{S_r}\to 0,
\]	
with
\[
\eF_k = \bigoplus_{\substack{n-d\ge\la_1\ge\dots\ge \la_r\ge 0\\ |\la|-\ell(\sigma)=k}} \Sigma_{\sigma(\la^*+\rho)-\rho} V_{n-d}^*\T\Sigma_{\la^t} \eU_d^\perp,
\]
where the tuples $\la$ which show up satisfy the condition that $\la^*+\rho$ has no coinciding entries.
\end{lem}
\begin{proof}
We use  the Koszul resolution \eqref{eq:Kosres}, the decomposition \eqref{eq:Lambda} and the Borel-Weil-Bott formula \eqref{eq:BWB}.
Let us show that the maximal possible value for $|\la|-\ell(\sigma)$ is equal to $r^2$. By definition,
\begin{equation}\label{eq:mixedweight}
\la^*+\rho = (n-d-\la_r,\dots, n-d-r+1-\la_1,n-d-r,\dots,1).
\end{equation}
Let
\[
\mu_1 = n-d- \la_r,\  
\mu_2 = n-d-1-\la_{r-1},\dots, \mu_r = n-d-r+1-\la_1.
\] 
We have the following properties:
\begin{itemize}
\item $n-d\ge \mu_1> \dots >\mu_r\ge -r+1$,
\item $\mu_i\notin \{n-d,\dots,1\}$, $i=1,\dots,r$.	
\end{itemize}	
We conclude that there are a total of  $\binom{2r}{r}$ terms (tensor products) in all $\eF_k$.

Let $a$ be the number of non-positive $\mu_i$'s. Then for the  permutation 
$\sigma\in S_{n-d}$ such that $\sigma (\la^*+\rho)$ is non-increasing
one has $\ell(\sigma)=a(n-d-r)$.
Our goal is to show that the possible values of $N=|\la|-\ell(\sigma)$ are from zero to $r^2$. 
Note that if $a=r$, then $\mu_i=-i+1$, $\la_i=n-d$ for all $i$ and hence $N=r^2$.
Now if one passes from $a$ to $a-1$, then the value of $\ell(\sigma)$ becomes smaller by $n-d-r$. 

Let $M_a(\la)$ be the maximal possible value of $|\la|$ for a fixed $a$; in particular, $M_r(\la)=r(n-d)$.  
Also let $m_a(\mu)$ be the minimal possible value of $|\mu|$ for a fixed $a$; in particular, $M_r(\la)+m_r(\mu)$ does not depend on $a$.
One easily sees
that  $m_a(\mu)$ is attained at
\[
\mu = (n-d+r-a,\dots, n-d+1,-r+a,\dots,-r+1).
\] 
Hence, $m_{a-1}(\mu) = m_a(\mu) + n-d-r+2(r-a)+1$ and 
$$M_{a-1}(\mu) = m_a(\mu) - (n-d-r+2(r-a)+1).$$
So the maximal possible value of $N=|\la|-\ell(\sigma)$ becomes 
smaller by $2(r-a)+1$ when we pass from $a$ to $a-1$. Hence the maximal 
value is $r^2$ showing up for $a=r$.
\end{proof}

\begin{rem}
As shown in the proof above, there are totally $\binom{2r}{r}$ terms in all $\eF_k$, $k=0,\dots,r^2$. Each term is a tensor product of an
irreducible finite-dimensional representation of $\mgl(V_{n-d})$ and 
a sheaf on $\Gr_d(V)$ which is equal to a Schur functor applied to $\eU_d^\perp$.  
\end{rem}

\begin{example}\label{ex:Fr}
One has 
\begin{enumerate}
\item $\eF_0 = \eO$, the term corresponds to $\la=0$,
\item $\eF_1 = \Sigma_{(0^{r-1}(-1)^{n-d-r+1})} V_{n-d}^*\T\Sigma_{(1^{n-d-r+1}0^{r-1})} \eU_d^\perp$, the term corresponds to $\la=(n-d-r+1,0^{r-1})$,
\item $\eF_{r^2} = \eO(-r)$, the term corresponds to $\la=((n-d)^r)$.
\end{enumerate} 
\end{example}

Our next goal is to compute the sections of the sheaf $\eI_{S_r}(1)$. 
We first formulate the following direct consequence of Lemma \ref{lem:complex} 
(recall that $\eF_0\simeq~\eO$).
\begin{cor}\label{cor:I(1)}
The sheaf $\eI_{S_r} (1)$ admits the following resolution
\[
0\to \eF_{r^2}(1)\to \dots \to \eF_1(1)\to \eI_{S_r} (1)\to 0.
\]
\end{cor}

Let us introduce the following notation
\[
\eT_\la = \Sigma_{\sigma(\la^*+\rho)-\rho} V_{n-d}^*\T\Sigma_{\la^t} \eU_d^\perp.
\]
In particular,
$\eF_k = \bigoplus_{\substack{n-d\ge\la_1\ge\dots\ge \la_r\ge 0\\ |\la|-\ell(\sigma)=k}} \eT_\la$.

\begin{lem}\label{lem:H^0}
$H^0(\eT_\la(1))\ne 0$ if and only if $\la=(n-d-m,0^{r-1})$ for $m=0,\dots,r-1$. The sheaf $\eT_{(n-d-m,0^{r-1})}$ is a summand of 
$\eF_{r-m}$. One has  
\[
H^0(\eT_{(n-d-m,0^{n-d-1})}(1))\simeq \Sigma_{(0^{r-1},(-1)^{n-d-r},m-r)}V^*_{n-d}\T \Lambda^m V.
\]
\end{lem}
\begin{proof}
By definition $H^0(\eT_\la(1))\ne 0$ if and only if $H^0(\Sigma_{\la^t} \eU_d^\perp(1))\ne 0$.  Recall that $\eO(1)=\Sigma_{(-1,\dots,-1)} \eU_d^\perp$. Hence 
\begin{equation}\label{eq:-1}
\Sigma_{\la^t} \eU_d^\perp(1) = \Sigma_{\la^t-(1,\dots,1)} \eU_d^\perp.
\end{equation}
By \eqref{eq:BWB} for a collection $\mu=(\mu_1\ge \dots\ge \mu_d)$ the space of sections of $\Sigma_\mu\eU^\perp$ is 
non-zero if and only if $\mu_1\le 0$. Hence the right hand side of \eqref{eq:-1} admits a non-trivial
section if and only if all the entries of $\la^t$ are at most one, so 
$\la$ has a single non-trivial part $\la_1$ (not exceeding $n-d$). 
The condition 
$n-d-r+1-\la_1\notin [n-d-r]$ implies $\la_1\ge n-d-r+1$.

The subscript $k$ of the sheaf $\eF_k$ containing the term  $\eT_{(n-d-m,0^{n-d-1})}$ is
computed as $k=|\la|-\ell(\sigma)=n-d-m-(n-d-r)=r-m$, which proves the second claim of the proposition. Finally, let us compute the corresponding spaces of sections.

For $\la=(n-d-m,0^{r-1})$ one has $\sigma(\la^*+\rho)-\rho=(0^{r-1},(-1)^{n-d-r},m-r)$.
Therefore
\[
\Sigma_{\sigma(\la^*+\rho)-\rho} V_{n-d}^* \simeq \Sigma_{(0^{r-1},(-1)^{n-d-r},m-r)}V^*_{n-d}.
\]
 Computing the sections of the corresponding sheaf one gets 
 \[
 H^0(\Sigma_{\la^t-(1,\dots,1)} \eU_d^\perp) = H^0(\Sigma_{0^{d-m},(-1)^m}\eU_d^\perp)
 \simeq \Lambda^m V. 
 \]	
\end{proof}

In Proposition \ref{prop:H0Gr} (see Appendix \ref{sec:app}) we use Lemma \ref{lem:H^0} 
to identify the dual
space of sections of $\eI_{S_r}(1)$ with the truncated wedge power $\Lambda^d(V)_r$.   
Recall the truncated Grassmannians
\[
X_r=\overline{\exp(\fa_d)[w_r]}\subset \bP(\Lambda^d(V)_r).
\] 

\begin{thm}\label{thm:main}
The blow-up  $\mathrm{Bl}_{r}$ admits a closed embedding into  $\Gr_d(V)\times X_{r-1}$. 
The image of the embedding is the closure of the $\exp(\fa_d)$ orbit of the product of  
$V_d$ and $[w_{r-1}]$.  
\end{thm}
\begin{proof}
We want  to construct the $\exp(\fa_d)$ equivariant embedding 
$\mathrm{Bl}_r$ into $\Gr_d(V)\times \bP(\Lambda^dV)_{r-1}$. 
Let $\eI_r$ be the ideal sheaf of the subvariety $S_r$. 
We use the following standard construction involving the relative $\mathrm{Proj}$ construction.
Let 
\begin{equation}\label{eq:A}
A=\bigoplus_{k\ge 0} S^k(W^*\T\eO(-1)) = \bigoplus_{k\ge 0} S^k(W^*)\T\eO(-k).
\end{equation}
Then $\Gr_d(V)\times \bP(\Lambda^dV)_{r-1}\simeq \mathrm{Proj}_{\Gr_d(V)} A$, because	
$\bP(W)=\mathrm{Proj} \bigoplus_{k\ge 0} S^k(W^*)$ and twisting by $\eO(-k)$ we get the 
desired result.

Recall that by definition
\[
\mathrm{Bl}_r = 
\mathrm{Proj}_{\Gr_d(V)} \left(\eO\oplus \eI_{S_r} \oplus (\eI_{S_r})^2\oplus \dots\right).
\] 
Hence in order to obtain the desired embedding 
$\mathrm{Bl}_r\hookrightarrow  \Gr_d(V)\times \bP(\Lambda^dV)_{r-1}$ it suffices to
construct a surjective homomorphism of algebras
\begin{equation}\label{eq:mapproj}
	\bigoplus_{k\ge 0} S^k((\Lambda^dV)_{r-1}^*\T\eO(-1))\to \bigoplus_{k\ge 0} (\eI_{S_r})^k.
\end{equation}
Using Proposition \ref{prop:H0Gr} we obtain the surjective map 
$(\Lambda^dV)_{r-1}^*\T\eO(-1)\to \eI_{S_r}$, which implies the desired embedding. 

To finalize the proof we note that $\mathrm{Bl}_r$ is realized inside  
$\Gr_d(V)\times \bP(\Lambda^dV)_r$ and admits the action of the group $\exp(\fa_d)$
with an open dense orbit. In fact, the Grassmannian $\Gr_d(V)$ contains an open 
dense orbit $\exp(\fa_d)V_d$, which does not intersect with $\exp(\fa_d)$ invariant $S_r$. 
Therefore, the preimage 
of $\exp(\fa_d)V_d$ in the blow-up is an open dense $\exp(\fa_d)$ orbit, since the blow-up is
irreducible. Hence it suffices to prove that the preimage of $V_d$ in $\mathrm{Bl}_r$
is exactly $V_d\times [w_r]$.   

To this end we restrict the surjection \eqref{eq:mapproj} to the open cell 
$C=\exp(\fa_d)V_d\subset \Gr_d(V)$; in particular, $C$ is identified with $\fa_d$. 
Since $S_r$ does not intersect this cell, the right hand 
side is the direct sum  $\bigoplus_{k\ge 0} \eO(C)$ and the left hand side is given by
$\eO[C]\bigoplus_{k\ge 0} S^k(\Lambda^dV_r)^*$ (here $\eO(C)$ is just the algebras of functions on the 
affine space $C$). The surjection between these restrictions is $\eO(C)$ linear and hence 
determined by
the $\fa_d$-equivariant map $(\Lambda^dV_r)^*\to \eO(C)$ (coming from the $k=1$ part). 
The image of this map consists of functions of degree at most one with 
$(\Lambda^dV_r)^*\subset W^*$ mapping to the constants and $\Hom(V_{n-d},V_d)\subset (\Lambda^dV_r)^*$
mapping to the space of linear functions, i.e. $\xi$ is mapped to the function on $C$
taking value $\xi(w+c)$ for a point $c\in C$ (recall that $W\simeq \bK w\oplus\fa_d$). 
We conclude that the preimage of the point $V_d$
in the blow-up is mapped to itself times $[w_r]$ inside $\Gr_d(V)\times \bP(\Lambda^dV)_r$.  
\end{proof}

\section{Mixed case and collineations}\label{sec:coll}
In this section we introduce a family of pairwise birationally isomorphic $\bG_a^{d(n-d)}$ 
varieties which include the truncated Grassmannians $X_r$ and the blow-ups $\Bl_r$.
 In particular, the "largest" member of the family is smooth 
and serves as a desingularization for all other members. We first recall the construction 
of (partial) collineations \cite{L88,Th99,Vain84}.
 
\subsection{Collineations}\label{subsec:coll}
Let $E_1$ and $E_2$ be two vector spaces of dimension $d_1$ and $d_2$ (in what follows
we will only need  the case $d_1=d_2$). Let $\bs=(1 \le s_1< \dots < s_k \le \min(d_1,d_2))$
be a collection of integers. 
Let $\Hom^\circ(E_1,E_2)$ be the full rank homomorphisms from $E_1$ to $E_2$.
For  $\varphi\in\Hom^\circ(E_1,E_2)$
we consider the corresponding (non-zero) elements 
$\Lambda^{s_i}\varphi \in \Hom(\Lambda^{s_i} E_1,\Lambda^{s_i} E_2)$ (the coordinates
of $\Lambda^{s_i}$ are the minors of $\varphi$). We thus obtain the birational map
\[
F_\bs: \bP(\Hom(E_1,E_2)) \dashrightarrow \prod_{i=1}^k \bP(\Hom(\Lambda^{s_i} E_1,\Lambda^{s_i} E_2)).
\]
The space of partial collineations $\Col_\bs(E_1,E_2)$ is defined as the closure
of the image of $F_\bs$:
\[
\Col_\bs(E_1,E_2) = \overline{\{([\Lambda^{s_1}\varphi],\dots,[\Lambda^{s_k}\varphi]), \varphi\in\Hom^\circ(E_1,E_2)\}} 
\]
(as usual, for a vector $\varphi$ we denote by $[\varphi]$ the corresponding line in the
projective space).

\begin{example}
Here are several examples of the collineation varieties. 
\begin{enumerate}
\item 
If $k=1$, $s_1=1$, then  $\Col_\bs(E_1,E_2)=\bP(\Hom(E_1,E_2))$. 
\item If $k=1$, $s_1>1$, then
$\Col_\bs(E_1,E_2)$ is the projective spectrum of the algebra generated by the $s_1$ minors
of a $d_1\times d_2$ matrix \cite{BV88,BC01,dCEP80}.
\item If $k=1$, $d_1<d_2$ and $s_1=d_1$, then $\Col_\bs(E_1,E_2)\simeq \Gr_{d_1}(E_2)$.
 \item If $k=1$, $d_1=d_2$, $s_1=d_1-1$, then $\Col_\bs(E_1,E_2)\simeq \bP^{d_1^2-1}$, since 
 the $(d_1-1)\times (d_1-1)$ minors of a matrix are algebraically independent. 
\item If $k$ is arbitrary, $s_1=1$, then  there is a natural projection map
from $\Col_\bs(E_1,E_2)$ to $\bP(\Hom(E_1,E_2))$.
\item If $d_1=d_2$, $k=d_1-1$, $\bs=(1,\dots,d_1-1)$, then $\Col_\bs(E_1,E_2)$ is smooth 
and projects onto all other collineation varieties.
\end{enumerate}	
\end{example}

\subsection{Mixed truncated Grassmannians}
As above, we fix $d$ and $n$ such that $d\le n-d$ and consider
a collection  $\br=(1 < r_1<\dots<r_k\le d)$. 
Recall the chain of embedded Schubert varieties $\Gr_d(V)\supset S_2\supset \dots\supset S_d$.
We define the blow-up
$\mathrm{Bl}_\br \Gr_d(V)$ as follows. Let $\pi_r: \mathrm{Bl}_r\to \Gr_d(V)$ 
be the standard projection (as above, $\Bl_r = \mathrm{Bl}_{S_r}\Gr_d(V)$). 
 
\begin{dfn}\label{dfn:mixedbl}
The mixed blow-up  $\mathrm{Bl}_\br\subset \prod_{r\in\br} \mathrm{Bl}_r$
is the closure of the following set
\[
(x_1,\dots,x_k):\ \pi_{r_1}(x_1) = \dots = \pi_{r_k} (x_k),\ \pi_{r_i}(x_i) 
\notin \bigcup_{i=1}^k S_{r_i}. 
\]  
\end{dfn}
\begin{rem}
Since $r_1<\dots <r_k$, the union $\bigcup_{i=1}^k S_{r_i}$ is equal to $S_{r_1}$.
\end{rem}

The varieties $\mathrm{Bl}_\br$ admit the following description in terms of truncated 
Grassmannians $X_r\subset \bP(\Lambda^d(V)_r)$. 
Recall that $X_r$ 
is the closure of the $\exp(\fa_d)$ orbit through the point $[w_r]$ -- 
the image of the highest weight line
in $\bP(\Lambda^d(V))$. In particular, $X_d=\Gr_d(V)$, since $\Lambda^d(V)_d=\Lambda^d(V)$.

All the blow-ups $\Bl_\br$ admit the action of the parabolic
subgroup $P^-_d$.
\begin{thm}
The variety $\mathrm{Bl}_\br$ admits a  $P^-_d$ equivariant embedding
\begin{equation}\label{eq:remb}
\mathrm{Bl}_\br \hookrightarrow \Gr_d(V)\times X_{r_1-1}\times\dots\times X_{r_k-1};
\end{equation}
the image of the embedding is the orbit closure of the product 
of highest weight lines
\[
\mathrm{Bl}_\br\simeq \overline{\exp(\fa_d)(V_d\times [w_{r_1-1}]\times\dots\times [w_{r_k-1}])}.
\]
\end{thm}
\begin{proof}
By Theorem \ref{thm:main} one gets embedding \eqref{eq:remb} and the image coincides 
with the orbit closure. 
\end{proof}

\begin{cor}
All varieties $\Bl_\br$ share the same open part (the open $\exp(\fa_d)$ orbit), 
and hence are naturally birationally isomorphic.
\end{cor}

The following proposition describes the fibers of the natural projection map
$\pi_\br: \Bl_\br\to\Gr_d(V)$.  
Let $U\in \Gr_d(V)$ satisfies  $\dim U\cap V_{n-d}=m$ and let $j$ be the index such that 
$r_j\le m < r_{j+1}$. 

We denote by $\mathrm{pr}$ the projection map $V\to V_d$
whose kernel is $V_{n-d}$.
\begin{prop}\label{prop:rfibers}
Let $U\in\Gr_d(V)$ and $\dim U\cap V_{n-d} = m$. Then $\pi_\br^{-1}U$	
is isomorphic to the space of partial collineations 
\[
\Col_{(m+1-r_j,\dots,m+1-r_1)}(U\cap V_{n-d},V_d/\mathrm{pr}(U)).
\]
\end{prop}
\begin{proof}
By definition, the mixed blow-up $\Bl_\br$ sits inside the product of the Grassmannian 
$\Gr_d(V)$ and the blow-ups $\Bl_{r_i}$. For $U\in \Gr_d(V)$ as above, let $A_U\subset \Gr_d(V)$ 
be the standard open affine cell containing $U$, $A_U\simeq \Hom(U,V/U)$.   
Definition \ref{dfn:mixedbl} implies that $\pi_\br^{-1}(A_U)$ is the closure of the set 
of points of the form $(x,\pi_{r_1}^{-1}(x),\dots,\pi^{-1}_{r_m}(x))$ for 
$x\in A_U\setminus S_{r_1}$. We note that if $i>j$, i.e. $r_i>m$, then $A_U\cap S_{r_i} =\emptyset$
(since for any $U_1\in A_U$ one has $\dim U_1\cap V_{n-d}\le \dim U\cap V_{n-d}$).
We conclude that the fiber $\pi_\br^{-1}U$ sits inside the product of fibers 
$\pi_{\br_i}^{-1}U$ for $i=1,\dots,j$. Now 
Proposition \ref{prop:r-fibers} claims that each $\pi_{\br_i}^{-1}U$ is 
isomorphic $\Col_{(m+1-r_i)}U\cap V_{n-d},V_d/\mathrm{pr}(U)$ and this description emerges 
from the closure of the map $\varphi\mapsto \Lambda^{m+1-r_i}\varphi$ (for $\phi$ of high 
enough rank). Since $\Bl_\br$ is defined as the closure of the diagonally embedded 
$\Gr_d(V)\setminus S_{r_1}$, we arrive ate the desired statement.
\end{proof}

Now let us consider the maximal (complete) blow-up $\Bl_{(2,\dots,d)}$. By definition,
one gets the projection map $\Bl_{(2,\dots,d)}\to \Bl_\br$ for any collection $\br$.
The following proposition shows that this map is a desingularization.
To simplify the notation, we denote $\Bl_{(2,\dots,d)}$ by $\Bl_{\max}$,

\begin{prop}
The complete blow-up $\mathrm{Bl}_{\max}$ is smooth. 
\end{prop}
\begin{proof}
Consider the $P^-_d$ equivariant projection $\pi_{\max}:\Bl_{\max}\to\Gr_d(V)$. 
If $\Bl_{\max}$ is singular, then there exists a point $U$ such that $\dim U\cap V_{n-d}=d$
and a singularity shows up over $U$, i.e. in $\pi_{\max}^{-1}U$. In fact, for any point in
the Grassmannian the closure of its $P^-_d$ orbit intersects with the Schubert variety $S_d$
(note that $S_d=\{U\in\Gr_d(V):\ U\subset V_{n-d}\}$).

Now for a point $U\in S_d$ let us describe the preimage $\pi_{\max}^{-1}A_U$ of the standard
cell $A_U\subset \Gr_d(V)$ containing $U$. The cell $A_U$ is identified with the space of linear maps $\Hom(U,V/U)$. For $\psi\in \Hom(U,V/U)$ we denote by $\varphi$ the composition $\mathrm{pr}\circ \psi$;
in particular, $\varphi\in \Hom(U,V_d)$. Clearly, $\Hom(U,V/U)\simeq \Hom(U,V_d)\oplus 
\Hom(U,V_{n-d}/U)$, since $U\subset V_{n-d}$.
Now by Proposition \ref{prop:r-fibers} 
the preimage $\pi_{\max}^{-1}A_U$ sits inside 
$A_U\times \prod_{s=1}^{d-1} \bP(\Hom(\Lambda^sU,\Lambda^sV_d))$ as the closure
of the set of points $(\psi,[\varphi],[\Lambda^2\varphi],\dots,[\Lambda^{d-1}\varphi])$.
Let us consider the projection map 
$\pi_{\max}^{-1}A_U\to \Col_{(1,\dots,d-1)}(U,V_d)$, which forgets the first coordinate $\psi$. 
The fiber over 
a point $(p_1,\dots,p_{d-1})$, $p_s\in \bP(\Hom(\Lambda^sU,\Lambda^sV_d))$ is
isomorphic to $\Hom(U,V_{n-d}/U)\times \mathrm{span}(p_1)$. Hence $\pi_{\max}^{-1}A_U$ is 
a vector bundle  over a smooth base, so  $\pi_{\max}^{-1}A_U$ is smooth as well.
\end{proof}

\section{The general case}\label{sec:general}
In this section we describe the general picture for the blow-ups of flag varieties along Schubert
subvarieties. To simplify the exposition, we concentrate on the case of complete flag
varieties. One can similarly work out the case of Schubert varieties inside partial flags.
In particular, in the previous sections we considered the case of Grassmannians. 

Let $G$ be a simple simply-connected Lie group. We fix a Borel subgroup $B$ and the maximal 
torus $T\subset B$.
Let $F=G/B$ be the corresponding flag variety. Let $B_-\subset G$ be the opposite
Borel subgroup, which acts on $F$ with an open
dense orbit through the class of identity.
Let $W$ be the Weyl group of $G$; for $\sigma\in W$, let $p_\sigma\in F$ be the corresponding 
$T$ fixed point in the flag variety. We denote by $S^\circ(\sigma)\subset S(\sigma)\subset F$ 
the (opposite)
open and closed Schubert varieties (in particular, $\dim S(\sigma) = \dim F - \ell(\sigma)$, 
where $\ell(\sigma)$ is the length of $\sigma$). Hence, $S^\circ(\sigma) = {B_- p_\sigma}$, $S(\sigma)=\overline{S^\circ(\sigma)}$.  
Each open Schubert variety is an open affine cell and $F=\sqcup_{\sigma\in W} S^\circ(\sigma)$.
We are interested in the blow-ups $\Bl_{S(\sigma)} F$. 

Let $\fg$ be the Lie algebra of $G$ with the Cartan decomposition $\fg=\fn_-\oplus\fh\oplus\fn$
and let $\fb=\fn\oplus\fh$, $\fb_-=\fn_-\oplus\fh$ be the Borel subalgebras. 
For a dominant weight $\la\in\fh^*$,  
we denote by $L(\lambda)$ the irreducible highest 
weight $\fg$ module with a highest weight vector $v(\lambda)\in L(\la)$. In particular, 
$L(\la)=\U(\fn_-)v(\la)$ and $\fn. v(\la)=0$.
For an element $\sigma\in W$ we fix an extremal weight vector $v(\sigma\la)\in L(\la)$ of
weight $\sigma\la$ (the space of such vectors is one-dimensional for every $\sigma$). 
Let $D(\sigma\la)$ be the (opposite) Demazure module $D(\sigma\la) = \U(\fn_-)v(\sigma\la)$.

Now assume that $\la$ is regular, i.e. all fundamental weights show up as summands of $\la$. 
One has the following $G$-equivariant closed embeddings 
\begin{gather*}
	F=G/B \subset \bP(L(\la)),\ eB\mapsto [v(\la)],\qquad S(\sigma) \subset \bP(D(\sigma\la)),\ 
	p_\sigma\mapsto [v(\sigma\la)].
\end{gather*}
The following Lemma is standard (see e.g. \cite{BK05,Kum02}), but important for the following discussion.
\begin{lem}\label{lem:SchDem}
	One has $S(\sigma) = F\cap \bP(D(\sigma\la)) \subset \bP(L(\la))$.
\end{lem}

We introduce the following notation for the truncated representations
\begin{equation}
	L_\sigma(\la) = L(\la)/ D(\sigma\la).
\end{equation}
We denote the image of the highest weight vector $v(\la)$ in $L_\sigma(\la)$  by $v_\sigma(\la)$. 

Let $\eL(\la)$ be the line bundle on $F$ obtained as the pull 
back of $\eO(1)$ with respect to the embedding $F\subset \bP(L(\la))$.
We use the same notation $\eL(\la)$ for the restriction of $\eL(\lambda)$ to $S(\sigma)$.
Also let $\eI_\sigma$ be the ideal sheaf (on $F$) of $S(\sigma)$. 

\begin{lem}
	$L_\sigma(\la)$ is a cyclic $\fb_-$ module with the cyclic vector $v_\sigma(\la)$.
	One has the isomorphism $L_\sigma(\la)^*\simeq H^0(F,\eI_\sigma\T\eL(\lambda))$.	
\end{lem}
\begin{proof}
	The first claim follows from the surjection of $\fb_-$ modules $L(\la)\to L_\sigma(\la)$,
	sending $v(\la)$ to $v_\sigma(\la)$. This implies that $L_\sigma(\la) = \U(\fb_-)v_\sigma(\la)$.	
	
	To prove the second claim, we recall the isomorphisms
	\[
	L(\la)^*\simeq H^0(F,\eL(\la)),\ D(\sigma\la)^*\simeq H^0(S(\sigma),\eL(\la)). 
	\]
	Hence $L_\sigma(\la)^*$ is realized as kernel of the restriction map
	$H^0(F,\eL(\la))\to  H^0(S(\sigma),\eL(\la))$. By definition, the sections of 
	$\eI_\sigma\T\eL(\lambda)$ are exactly the sections of $\eL(\la)$ which vanish on 
	the Schubert variety $S(\sigma)$. Hence, we obtain the second claim.
\end{proof}

\begin{cor}
	The blow-up of $\bP(L(\la))$ along $\bP(D(\sigma\la))$ is isomorphic to the product
	$\bP(L(\la))\times \bP(L_\sigma(\la))$. The blow-up consists of pairs of lines $(\ell_1,\ell_2)$
	such that $\ell_1$ projects to $\ell_2$ under the projection map 
	$L(\la)\to L_\sigma(\la)$.
\end{cor}
\begin{proof}
	We use the standard construction of the blow-up of a projective space along a subspace 
	(see e.g. \cite{Vak25}) and definition of $L_\sigma(\la)$.
\end{proof}

In order to describe the restriction of the above blow-up,  we define the truncated 
flag variety
\begin{equation}\label{eq:trflag}
	X_\sigma(\lambda) = \overline{B_- v_\sigma(\la)} \subset \bP(L_\sigma(\la)).
\end{equation}
We define the variety $F_\sigma(\la)$ inside the product $F\times X_\sigma(\la)$ 
of the flag variety and the truncated flag variety by 
\[
F_\sigma(\la) = \overline{B_- \bigl([v(\la)]\times [v_\sigma(\la)]}\bigr)
\subset F \times X_\sigma(\la) \subset \bP(L(\la))\times \bP(L_\sigma(\la)). 
\]
By definition, there is a canonical projection $F_\sigma(\la)\to F = G/B$. 

\begin{thm}
	The blow-up of $F$ along $S(\sigma)$ is embedded into $\bP(L(\la))\times \bP(L_\sigma(\la))$
	as a strict transform of $F\subset \bP(L(\la))$.  The image of the embedding is equal
	to $F_\sigma(\la)$. 
\end{thm}
\begin{proof}
	By Lemma \ref{lem:SchDem} we know that the intersection of $F$ with $\bP(D(\sigma\la))$
	is $S(\sigma)$. Hence $\Bl_{S(\sigma)}F$ is the closure of the preimage of 
	$F\setminus S(\sigma)$ inside $\bP(L(\la))\times \bP(L_\sigma(\la))$ under the natural
	projection map from $\Bl_{\bP(D(\sigma\la))} \bP(L(\la))$ to $\bP(L(\la))$.  	
	
	Now let us prove the second claim. The orbit $B_- \bigl( [v(\la)]\times [v_\sigma(\la)]\bigr)$ 
	is contained in the blow-up of $\bP(L(\la))$ along  $\bP(D(\sigma\la))$. Since the $B_-$
	orbit of $[v(\la)]$ is open in $F$ and does not intersect the Schubert variety $S(\sigma)$, 
	the orbit $B_- \bigl([v(\la)]\times [v_\sigma(\la)]\bigr)$	is contained in $\Bl_{S(\sigma)}F$.
	Since the blow-up is irreducible, we conclude that $F_\sigma(\la)$  coincides with
	the blow-up of $F$ along $S(\sigma)$.
\end{proof}

It is natural to generalize the construction above to the case of multiple permutations.
Namely, for a collections $\underline\sigma=(\sigma_1,\dots,\sigma_m)\in W^m$
we define the orbit closure
\[
F_{\underline\sigma}(\la) = \overline{B_- \times_{i=1}^m [v_{\sigma_i}(\la)]} \subset 
\prod_{i=1}^m X_{\sigma_i}(\la). 
\]
It is tempting to conjecture that if $\underline\sigma$ exhausts the set of all Weyl group elements, 
then $F_{\underline\sigma}(\la)$ is smooth.

\begin{rem}
One can similarly consider the case of a partial flag variety $F=G/P$ for a standard 
parabolic subgroup $P$. The difference is that the weights $\lambda$ providing the
embedding $F\subset \bP(L(\la))$ do not have to be regular (only $P$-regular) and
there are less Schubert varieties (since part of the Weyl group stabilizes the 
point corresponding to the identity). 
\end{rem}

Finally, let us establish a correspondence between the objects studied in the previous sections
for the Grassmannians and the general objects introduced above. The relevant flag variety
$F$ is $SL_n/P_d$, $SL_n=SL(V)$, for a maximal parabolic $P_d$. We fix $\la=\om_d$; then
 $L(\la)=\Lambda^d(V)$  and the corresponding projective
embedding $\Gr_d(V)\subset \bP(L(\lambda))$ is the classical Pl\"ucker embedding.
For $r=1,\dots,d$ we fix permutation $\sigma_r$ such that
$\sigma_r([d])=\{1,\dots,d-r,d+1,\dots,d+r\} = I(r)$ (see \eqref{eq:I(r)}).
Then the Demazure module $D(\sigma_r\om_d)\subset \Lambda^d V$ is the span of wedge products $v_I$ 
such that $\deg I\ge r$ (i.e. $|I_{>d}|\ge r)$; the cyclic vector of $D_r$ is $v_{I(r)}$.
The Schubert variety $S_r=S(\sigma_r)$ consists of 
subspaces $U\in \Gr_d(V)$ such that $\dim U\cap V_{n-d}\ge r$
(see Lemma \ref{lem:dimint}). 
Now the truncated representation $L_{\sigma_r}(\om_d)=L(\om_d)/D(\sigma_r\om_d)$ is
exactly $(\Lambda^dV)_{r-1}$ and $X_{r-1}=F_{\sigma_r}(\om_d)$.

\section*{Acknowledgments}
We are grateful to Alexander Kuznetsov for generously sharing his ideas and for 
patient explanations. We are also grateful to Viktoriia Borovik and Svala Sverrisd\'ottir
for useful discussions about graph closures.
This work was partially supported by the ISF grant 493/24.

\appendix 
\section{Twisted ideal sheaf}\label{sec:app}
In this section we provide explicit computation of the space of sections of the twisted 
ideal sheaf of the Schubert variety $S_r\subset\Gr_d(V)$.  
We identify $H^0(\eI_{S_r}(1))$ with 
the dual of the truncated module $(\Lambda^d V)_r$. We prepare the following lemma.

\begin{lem}\label{lem:surj}
	The $\fp^-_{d}$ module $\Lambda^r V_{n-d}^*\T \Lambda^r V$ surjects onto  $(\Lambda^d V)^*_r$.
\end{lem}
\begin{proof}
	Let us show that there exists an embedding  $(\Lambda^d V)_r\subset \Lambda^r V_{n-d}\T \Lambda^r V^*$. We rewrite
	\[
	\Lambda^r V_{n-d}\T\Lambda^r V^* \simeq \bigoplus_{k=0}^r \left(\Lambda^r V_{n-d}\T\Lambda^k V_{n-d}^*\right)\T \Lambda^{r-k} V_d^*.
	\]
	The tensor product $\Lambda^r V_{n-d}\T\Lambda^k V_{n-d}^*$ contains 
	$\Lambda^{r-k} V_{n-d}$ as a direct summand in the decomposition into irreducible 
	$\msl(V_{n-d})$ modules.  
	Hence we obtain a $\msl(V_{n-d})\oplus\msl(V_d)$ submodule 
	$\bigoplus_{k=0}^r \Lambda^{r-k} V_{n-d}\T \Lambda^{r-k} V_d^*$, which is 
	isomorphic to the restriction of $(\Lambda^dV)_r$ from $\fp^-_{d}$ to its Levi subalgebra $\msl(V_{n-d})\oplus\msl(V_d)$ of $(\Lambda^d V)_r$. Hence its suffices to prove that 
	$\bigoplus_{k=0}^r \Lambda^{r-k} V_{n-d}\T \Lambda^{r-k} V_d^*$ is a $\fp^-_{d}$  submodule of 
	$\Lambda^r V_{n-d}\T \Lambda^r V^*$.
	
	Let us consider the trivial subrepresentation inside $\Lambda^r V_{n-d}\T \Lambda^r V_d^*$
	(which sits inside $\Lambda^r V_{n-d}\T \Lambda^r V^*$). We denote by $u_r$ a generator of this
	one-dimensional $\msl(V_{n-d})\oplus\msl(V_d)$ module. Since  $(\Lambda^d V)_r$ is cyclic 
	$\fa_d$ module, we need to show that the universal  enveloping algebra $\U(\fa_d)$ generate 
	from $u_r$ a submodule isomorphic to $(\Lambda^d V)_r$. 
	
	First note that $\fa_d$ preserves $V_{n-d}$ and maps (via the dual action) $V_{n-d}^*$ to $V_d^*$.
	Hence 
	\[
	\fa_d: \Lambda^r V_{n-d} \T \left(\Lambda^k V_{n-d}^* \T \Lambda^{r-k} V_d^*\right) \to
	\Lambda^r V_{n-d} \T \left(\Lambda^{k-1} V_{n-d}^* \T \Lambda^{r-k+1} V_d^*\right).
	\]
	Since $\fa_d$ is commutative, the universal enveloping algebra of $\fa_d$ is just the 
	polynomial ring. Let us show that applying degree $k$ part $\U(\fa_d)_k$ 
	(the degree $k$ polynomials in a basis of $\fa_d$) 
	to $u_r$ we get exactly the $\msl(V_{n-d})\oplus\msl(V_d)$ module 
	$\Lambda^k V_{n-d}\T \Lambda^k V_d^*$.  
	
	Let us consider the tensor product $\Lambda^r V_{n-d}\T \Lambda^{r-k} V_{n-d}^*$. The 
	$\mgl(V_{n-d})$ module $\Lambda^r V_{n-d}$ corresponds to the partition $(1^r)$; the module
	$\Lambda^{r-k} V_{n-d}^*$ corresponds to the partition $(1^{n-d-r+k})$. By the 
	Littelwood-Richardson rule the tensor product of these modules decomposes into several summands
	with one of them corresponding to the partition $(2^k1^{n-d})$. All other summands are labeled
	by partitions whose number of parts is strictly less than $n-d$ and the sum of parts 
	equal to $n-d+k$. We note that non of these summands tensored by $\Lambda^k V_d^*$ 
	may show up in the decomposition of $\U(\fa_d)_ku_r$ into irreducible $\msl(V_{n-d})\oplus\msl(V_d)$
	modules, since $\fa_d\simeq V_{n-d}\T V_d^*$ and $u_r$ generates a trivial module; hence
	in the decomposition of $\U(\fa_d)_ku_r$ into irreducible $\msl(V_{n-d})$ modules one sees
	only representations labeled by partitions with the sum of the parts equal to $k$. 
	Now getting back to the module corresponding to $(2^k1^{n-d})$, we see that as $\msl(V_{n-d})$ 
	module it is isomorphic to $\Lambda^k V_{n-d}$. Hence 
	\[
	\U(\fa_d)u_r\simeq (\Lambda^d V)_r\subset \Lambda^r V_{n-d}\T \Lambda^r V^*.
	\]      
\end{proof}

Now we are ready to prove the following proposition.

\begin{prop}\label{prop:H0Gr}
	The sheaf $\eI_{S_r} (1)$ is globally generated and one has an isomorphism of $\fp^-_{d}$ modules 	
	\[
	H^0(\Gr_d(V),\eI_{S_r}(1))\simeq (\Lambda^d V)^*_r.
	\]
\end{prop}
\begin{proof}
	Corollary \ref{cor:I(1)} and Example \ref{ex:Fr} give a surjection 
	\begin{multline*}
		\eF_1(1) \simeq \eT_{(n-d-r+1,0^{r-1})}(1) \simeq\\ 
		\Sigma_{(0^{r-1},(-1)^{n-d-r+1})}V_{n-d}^*\T \Sigma_{(0^{r-1},(-1)^{n-d-r+1})}\eU^\perp_d
		\to \eI_{S_r}(1).
	\end{multline*}
	Since 	$\Sigma_{(0^{r-1},(-1)^{n-d-r+1})}\eU^\perp_d$ is globally generated, the twisted ideal sheaf 
	is globally generated as well. We also get a surjection on the level of sections
	\begin{equation}\label{eq:surjsec}
		H^0(\eF_1)\simeq \Lambda^{r-1}V^*_{n-d}\T \Lambda^{r-1} V \to H^0(\eI_{S_r}(1)),
	\end{equation}
	since 	$(0^{r-1},(-1)^{n-d-r+1})=((-1)^{n-d})+(1^{r-1},0^{n-d-r+1})$ and adding
	weight $((-1)^{n-d})$ does not change the restriction of the corresponding $\mgl(V_{n-d})$ module to $\msl(V_{n-d})$. 
	Lemma \ref{lem:surj} implies that in order to complete the proof it suffices 
	to show that the kernel of the map
	\eqref{eq:surjsec} coincides with the kernel of map 
	$\Lambda^{r-1}V^*_{n-d}\T \Lambda^{r-1} V \to (\Lambda^d V)^*_{r-1}$ 
	
	Lemma \ref{lem:H^0} and Corollary \ref{cor:I(1)} imply that there exists the following exact sequence
	\begin{multline*}
		0\to \Sigma_{(0^{r-1},(-1)^{n-d-r},1-r)}V^*_{n-d}\T  V \to 
		\Sigma_{(0^{r-1},(-1)^{n-d-r},2-r)}V^*_{n-d}\T \Lambda^2 V \to\\
		\dots\to
		\Sigma_{(0^{r-1},(-1)^{n-d-r+1})}V^*_{n-d}\T \Lambda^{r-1} V \to H^0(\eI_{S_r}(1))\to 0.	
	\end{multline*}	
	The $\mgl(V_{n-d})$ module $\Sigma_{(0^{r-1},(-1)^{n-d-r},s-r)}V^*_{n-d}$ is isomorphic to 
	the irreducible representation $\Sigma_{((r-s),1^{n-d-r})}V_{n-d}$.
	Hence the terms in the above exact sequence are written as (here $s=0,\dots,r-1$)
	\[
	\Sigma_{((r-s),1^{n-d-r})}V_{n-d} \T \Lambda^s V \simeq 
	\bigoplus_{m=0}^s \left(\Sigma_{((r-s),1^{n-d-r})}V_{n-d} \T \Lambda^{s-m} V_{n-d}\right) \T \Lambda^mV_d.
	\]
	Let us denote the right hand side by $M_s$; hence, $M_s$ are $\msl(V_{n-d})\oplus\msl(V_d)$
	modules. 
	We want to compute the Euler characteristic of $[M_s]$ (starting from $s=r-1$), where 
	for a $\msl(V_{n-d})\oplus\msl(V_d)$ module $M$ we denote by $[M]$ its element 
	in the Grothendieck ring, i.e. a formal linear combination of its irreducible summands.
	Our goal is to show that 
	\begin{equation}\label{eq:Groth}
		\sum_{s=0}^{r-1} [M_s](-1)^{r-1-s} = \bigoplus_{m=0}^{r-1} \Lambda^m V_{n-d}^*\T \Lambda^m V_d. 	
	\end{equation}	 
	Clearly, each $M_s$ is a direct sum of terms $R_m\T \Lambda^m V_d$ for some $\msl(V_{n-d})$ 
	representations $R_m$. One has 
	\begin{multline}\label{eq:Rm}
		[R_m] = [\Sigma_{(1^{n-d-r+1})}\T \Sigma_{(1^{r-m-1})}] - [\Sigma_{(21^{n-d-r})}\T \Sigma_{(1^{r-m-2})}] +\\
		[\Sigma_{(31^{n-d-r})}\T \Sigma_{(1^{r-m-3})}] -\dots  + (-1)^{r-m+1}
		[\Sigma_{((r-m),1^{n-d-r})}\T \Sigma_{(0)}],
	\end{multline} 
	where all the Schur functors are applied to $V_{n-d}$ and we omit zeroes at the end of partitions
	(to simplify the formula).
	We compute \eqref{eq:Rm} using the Littelwood-Richardson rule (to be precise, we only need
	the Pieri formulas).
	
	All the summands showing up in \eqref{eq:Rm} are the tensor products of a thin hook and 
	a one column partition; in particular, one of the factors is always a wedge power of $V_{n-d}$.
	Hence, applying the Pieri formula, one gets a linear combination of 
	(classes of) representations labeled by partitions of the form $(N,2^a,1^b)$ (as above,
	zeroes at the end are omitted). For each $(N,a,b)$ with $N>1$ the corresponding summand
	shows up twice in \eqref{eq:Rm} in two adjacent terms and hence cancels. For $N=1$ (and, hence,
	$a=0$) the corresponding term show up only in the tensor product
	\[
	[\Sigma_{(1^{n-d-r+1})}\T \Sigma_{(1^{r-m-1})}],\ n-d-m = b+1
	\]
	with a positive sign. We conclude that $[R_m]=[\Sigma_{(1^{n-d-m},0^m)}V_{n-d}] = 
	[\Lambda^m V_d^*]$ and hence $H^0(\eI_{S_r}(1))\simeq (\Lambda^dV)_r^*$.
\end{proof}

\end{document}